\documentclass[11pt]{article}

\usepackage{listings}
\usepackage{epsfig}
\usepackage{lscape}
\usepackage{multirow}
\usepackage{longtable}
\usepackage{amsmath,amssymb,amsthm}
\usepackage{graphicx}
\usepackage{color,xcolor}
\usepackage{placeins}
\usepackage{url}
\usepackage{cases}
\usepackage{hyperref}
\usepackage{algorithm}  
\usepackage{setspace}
\usepackage{subcaption}
\usepackage{ulem}
\usepackage[shortlabels]{enumitem}
\usepackage{algpseudocode}
\counterwithout{figure}{section}

\usepackage{tikz}
\usetikzlibrary{
intersections, arrows.meta,
automata,er,calc,
backgrounds,
mindmap,folding,
patterns,
decorations.markings,
fit,snakes,
shapes,matrix,
positioning,
shapes.geometric,
arrows,through
}

\usepackage{booktabs}

\DeclareMathOperator*{\argmin}{argmin}

\newcommand{\red}[1]{\begin{color}{red}#1\end{color}}

\newtheorem{theorem}{Theorem}
\newtheorem{proposition}{Proposition}
\newtheorem{lemma}{Lemma}

\newtheorem{remark}{Remark}
\newtheorem{assumption}{Assumption}

\begin{document}
	
	\title{\bf An active-set based recursive approach for solving convex isotonic regression with generalized order restrictions}

	\author{
		{Xuyu Chen}\thanks{School of Mathematical Sciences, Fudan University, Shanghai, 200433, China, chenxy18@fudan.edu.cn}, \;
		Xudong Li\thanks{School of Data Science, Fudan University, Shanghai, 200433, China, lixudong@fudan.edu.cn}, \;
		{Yangfeng Su}\thanks{School of Mathematical Sciences, Fudan University, Shanghai, 200433, China, yfsu@fudan.edu.cn}
	}

	\date{\today}
	\maketitle
	
	\begin{abstract}
	This paper studies the convex isotonic regression with generalized order restrictions induced by a directed tree. 
	The proposed model covers various intriguing optimization problems with shape or order restrictions, including the generalized nearly isotonic optimization and the total variation on a tree. Inspired by the success of the pool-adjacent-violator algorithm and its active-set interpretation, we propose an active-set based recursive approach for solving the underlying model. 
	Unlike the brute-force approach that traverses an exponential number of possible active-set combinations, our algorithm has a polynomial time computational complexity under mild assumptions. 	
	
	\end{abstract}

\section{Introduction}
Given a directed tree $G=(V,E)$, we consider the following convex isotonic regression problem with generalized order restrictions:
\begin{equation}
\label{prob:tree_gnio}
\min_{x\in\Re^{|V|}} \; \sum_{i \in V} f_i(x_i) + \sum_{(i,j) \in E} \lambda_{{i,j}} (x_i - x_{j})_{+} + \sum_{(i,j) \in E} \mu_{i,j} (x_{j} - x_{i})_{+},
\end{equation}
where  for each $i \in V,$ $f_i:\Re \to \Re$ is a convex loss function, $\lambda_{i,j}$ and $\mu_{i,j}$ for $(i,j) \in E,$ are possibly  infinite nonnegative scalars, i.e., $0\le \lambda_{i,j}, \mu_{i,j} \le +\infty$, and $(x)_{+} = \max(0,x)$ is the nonnegative part of $x$ for any $x \in \Re$.
In \eqref{prob:tree_gnio}, when $\lambda_{i,j} = +\infty$ (respectively, $\mu_{i,j} = +\infty$), the corresponding term $\lambda_{i,j}(x_{i} - x_{j})_+$ (respectively,  $\mu_{i,j}(x_{j} - x_{i})_+$) should be understood as the indicator function $\delta(x_i, x_{j} \mid x_i - x_j \le 0)$ (respectively, $\delta(x_i, x_{j} \mid x_{i} -  x_{j} \ge 0)$), or equivalently the constraint $x_i - x_{j} \le 0$ (respectively, $x_{i} -  x_{j} \ge 0$). 
See Figure \ref{fig:trees} for some simple examples of directed trees.
\begin{figure}[htbp]
	\centering
	\begin{subfigure}{0.32 \linewidth}
		\parbox[][2.2cm][c]{\linewidth}{
			\centering
		\begin{tikzpicture}[shorten >= 1pt]
		\tikzstyle{vertex} = [circle,fill=black!25,minimum size=17pt,inner sep=0pt]
        \node[vertex] (1) { $1$ };
        \node[vertex] (2) [right of = 1] { $2$ };
        \node (3) [right of = 2] {$\cdots$};
        \node[vertex] (n) [right of = 3] { $n$};

        \draw[->]  (1) -- (2);
        \draw[->]  (2) -- (3);
        \draw[->]  (3) -- (n);
        \hfill
		\end{tikzpicture}
		}
		\caption{chain}
		\label{fig1a:chain}
	\end{subfigure}
	\centering
	\begin{subfigure}{0.32 \linewidth}
	\parbox[][2.2cm][c]{\linewidth}{
		\centering
		\begin{tikzpicture}[shorten >= 1pt]
		\tikzstyle{vertex} = [circle,fill=black!25,minimum size=17pt,inner sep=0pt]
        \node[vertex] (1) { $1$ };
        \node[vertex] (2) [below right = 1] { $2$ };
        \node[vertex] (3) [below left =  1] { $3$ };
        \node[vertex] (4) [below right = 2] { $4$ };
        \node[vertex] (5) [below left of = 3] { $5$ };
        \node[vertex] (6) [below right of = 3] { $6$ };
        \draw[->] (1) -- (2);
        \draw[->] (1) -- (3);
        \draw[->] (3) -- (5);
        \draw[->] (2) -- (4);
        \draw[->] (3) -- (6);
		\end{tikzpicture} }
		\caption{arborescence}
		\label{fig1b:rooted}
	\end{subfigure}
	\centering
	\begin{subfigure}{0.32 \linewidth}
	\parbox[][2.2cm][c]{\linewidth}{
		\centering
		\begin{tikzpicture}[shorten >= 1pt]
		\tikzstyle{vertex} = [circle,fill=black!25,minimum size=17pt,inner sep=0pt]
        \node[vertex] (1) { $1$ };
        \node[vertex] (2) [below left = 1] { $2$ };
        \node[vertex] (3) [below right =  1] { $3$ };
        \node[vertex] (4) [below left of = 3] { $4$ };
        \node[vertex] (5) [below right of = 3] { $5$ };

        \draw[->] (2) -- (1);
        \draw[->] (3) -- (1);
        \draw[->] (3) -- (5);
        \draw[->] (3) -- (4);

		\end{tikzpicture} 
		}
		\caption{general directed tree}
		\label{fig1c:directed}
	\end{subfigure}
	\caption{Examples of directed trees. A directed tree is a directed graph whose underlying graph is a tree, and the directed trees are also referred to as directed acyclic graphs.}
	\label{fig:trees}
\end{figure}
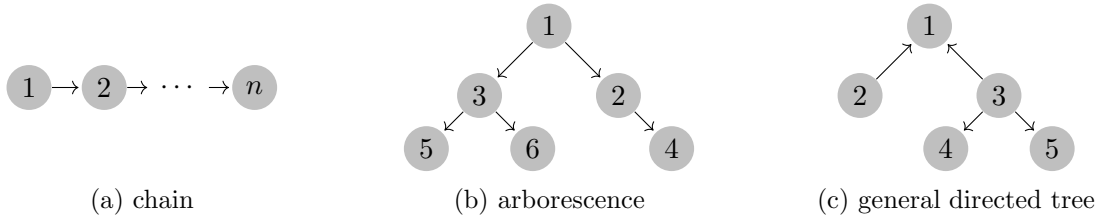

As one can observe, the involvement of the directed tree $G$ makes problem \eqref{prob:tree_gnio} a rather general model containing many interesting variants as special cases. Here, for simplicity, we only mention two of them.
The first one is the \textit{generalized nearly isotonic optimization} (GNIO) problem proposed in \cite{chen2021dynamic}:
\begin{equation} 
\label{prob:gnio}
\min_{x\in\Re^n} \; \sum_{i=1}^{n} f_i(x_i) + \sum_{i=1}^{n-1} \lambda_{i} (x_i - x_{i+1})_{+} + \sum_{i=1}^{n-1} \mu_{i} (x_{i+1} - x_{i})_{+},
\end{equation}
which is clearly a special case of \eqref{prob:tree_gnio} with $G$ chosen as a chain, as illustrated in Figure \ref{fig1a:chain}.
As is mentioned in \cite{chen2021dynamic}, model \eqref{prob:gnio} recovers, as special cases, many
classic problems in shape restricted statistical regression, including isotonic regression \cite{brunk1955maximum,
chakravarti1989isotonic}, unimodal regression \cite{frisen1986unimodal,stout2008unimodal}, and nearly isotonic regression \cite{tibshirani2011nearly}.
The second one is the \textit{total variation on a tree} considered in \cite{kolmogorov2016total}:
\begin{equation}
\label{prob:tv_trees}
\min_{x\in\Re^{|V|}} \; \sum_{i\in V} f_i(x_i) + \sum_{(i,j)\in E} w_{i,j} |x_i - x_{j}|, 
\end{equation}
where $G = (V,E)$ is a directed tree and each $f_i$ is assumed to be piecewise linear or piecewise quadratic.
Other special cases of model \eqref{prob:tree_gnio} have also been examined in the literature, for example, \cite{chakravarti1992isotonic,yu2016exact} studied the isotonic regression problems with partial order restrictions induced by an arborescence. These special cases, as well as their applications in statistic inference \cite{Silvapulle2005constrained}, operations research \cite{ayer1955empirical}, signal processing \cite{Alfred1993localiso,xiaojun2016videocut}, medical prognosis \cite{Ryu2004medical},
and traffic and climate data analysis \cite{Matyasovszky2013climate,Wu2015traffic}, reveal the importance and necessity of studying model \eqref{prob:tree_gnio}.

To the best of our knowledge, there is currently no efficient algorithm available for directly solving the general model \eqref{prob:tree_gnio}. 
However, certain special cases of the model can be solved by  existing algorithms. For example, the GNIO problem \eqref{prob:gnio} can be efficiently solved by employing a dynamic programming approach designed in \cite{chen2021dynamic}.
Moreover, assuming boundedness of the decision variables, the KKT based fast algorithm proposed in \cite{hochbaum2021unified} can also solve the GNIO problem. However, both algorithms rely heavily on the underlying chain structure, and therefore cannot be applied to solve the general model \eqref{prob:tree_gnio} that involves a directed tree. If $G$ is a chain and each $f_i$ is quadratic, the total variation problem \eqref{prob:tv_trees} reduces to the well-known $\ell_2$ total variation denoising problem, which has been extensively studied in signal processing \cite{condat2013direct, hoefling2009path}. The direct algorithm \cite{condat2013direct} and the taut-string algorithms \cite{barbero2018modular} are considered to be the state-of-the-art for solving the $\ell_2$ total variation denoising problem.
Meanwhile, if $G$ is assumed to be a directed tree and each $f_i$ is assumed to be continuous piecewise linear or piecewise quadratic with a finite number of breakpoints in \eqref{prob:tv_trees},
the message passing algorithm studied in \cite{kolmogorov2016total} can be applied. 
However, these algorithms can not handle problem \eqref{prob:tree_gnio} with general convex loss functions $f_i$ involved.

There is also another line of work dedicated to solving special cases of problem \eqref{prob:tree_gnio}. 
In the 1950s, Ayer in \cite{ayer1955empirical}  proposed the famous \textit{Pool-Adjacent-Violator algorithm} (PAVA) for solving the following isotonic regression problem:
\begin{equation}
\label{prob:iso}
\begin{aligned}
&\min_{x \in \Re^n} \quad \frac{1}{2} \sum_{i=1}^n  (x_i - y_i)^2, \\
& \quad \text{s.t.} \quad x_{1}  \le x_{2}\le \ldots \le x_{n},
\end{aligned}
\end{equation}
which is clearly a special case of problem \eqref{prob:tree_gnio}.
The PAVA has been widely regarded as the state-of-the-art technique for solving the isotonic regression problem since its inception. 
Later in \cite{best1990active}, Best and Chakravarti  discovered that the PAVA is, in fact, a dual feasible active set method for solving \eqref{prob:iso}. In \cite{best2000minimizing}, the PAVA was generalized to handle \eqref{prob:iso} but with the least squares objectives replaced by general separable convex loss functions. In \cite{yu2016exact}, Yu and Xing further generalized the PAVA to solve convex separable minimization with order constraints induced by an arborescence. However, the generalized regularizers present in the objective of model \eqref{prob:tree_gnio} were not studied in \cite{yu2016exact}. As far as we know, it remains unclear whether the ideas behind the PAVA can be adopted to solve the more general model \eqref{prob:tree_gnio}.

Encouraged by the successes of the PAVA and its variants in solving special cases of the generalized convex isotonic regression problem \eqref{prob:tree_gnio}, we propose a novel active-set based algorithm in this paper.
Our approach differs from the brute-force method that explores a potentially exponential number of different active sets. 
Instead, a recursive approach is proposed to accelerate the search for the desired active sets.  We show that problem \eqref{prob:tree_gnio} can be tackled via recursively solving a sequence of smaller subproblems. For these subproblems, special recursive structures of the corresponding Karush-Kuhn-Tucker (KKT) conditions are carefully examined, which further allows us to design a novel active-set based recursive approach (ASRA). In particular, this approach enables us to derive semi-closed formulas of the  optimal solutions to the aforementioned recursive subproblems. Under mild assumptions, we further show that the ASRA enjoys a polynomial time computational complexity for solving problem \eqref{prob:tree_gnio}.

The subsequent sections of this paper are organized as follows. 
Section \ref{sec:preliminary} covers the necessary preliminaries associated with problem \eqref{prob:tree_gnio}, including fundamental concepts in graph theory and the corresponding KKT conditions. In addition, we describe a naive active-set method to solve  \eqref{prob:tree_gnio}. Our recursive approach, the ASRA, is described in detail in Section \ref{sec:algorithm}. Finally, we conclude the paper in Section \ref{sec:conclusion}. The Appendix includes an example of how to apply the ASRA to solve a simple instance of \eqref{prob:tree_gnio}.

\section{Preliminaries}
\label{sec:preliminary}
We start with some relevant preliminaries in graph theory.
A directed tree $G = (V,E)$ is a directed graph whose underlying graph is a tree, and an arborescence (also known as rooted directed tree)  \cite{deo1974graph,west2001introduction} is a directed tree with exactly one node of zero in-degree.
The node is also referred to as the \textit{root} of the arborescence. 
Let $G = (V,E)$ and $B = (V_B, E_B)$ be two directed trees. If $V_B \subseteq V$ and $E_B \subseteq E$, then we say that $B$ is a \textit{subtree} of $G$, denoted by $B \subset G$.
Two subtrees are \textit{disjoint} if their node sets are disjoint. 
Given $P = \{B_k \}_{k=1}^K$ as a collection of disjoint subtrees of a certain directed tree $G = (V,E)$, if $V = \cup_{k=1}^K V_{B_k}$, then $P$ is said to be a \textit{partition} of $G$.

For a given directed tree $G=(V,E)$, we can choose any node $l\in V$ as the {\it ancestor} of $G$.
Then, for any $i,j\in V$, we say that $j$ is a child of $i$, denoted by $j \triangleleft i$, if the undirected path connecting $l$ and $i$ is strictly contained in the one connecting $l$ and $j$.
For example, if we pick the node $2$ as the ancestor in the directed tree presented in Figure \ref{fig1c:directed}, then we have $3 \triangleleft 1 \triangleleft 2$.
Now, let $D \in \Re^{|V| \times |E|}$ be the node-arc incidence matrix associated with $G$. We know from \cite{bertsimas1997introduction} that ${\rm rank}(D) = |E|$ and the matrix $\tilde D_l \in \Re^{|E|\times |E|}$ obtained by deleting the $l$-th row from $D$ is invertible. Given a vector $b\in\Re^{|E|}$, we obtain in the following lemma {a closed-form} formula for the solution to the linear system $\tilde D_l z = b$.
\begin{lemma}
\label{lem:na_sol}
For any given $b\in \Re^{|E|}$, the unique solution $z^* = (z_{i,j})_{(i,j)\in E} \in \Re^{|E|}$ to the linear system $\tilde D_l z = b$ takes the following form:
\begin{equation*}
	z^*_{i,j} = 
	\left\{ 
	\begin{aligned}
		& \sum_{k\in C_i} b_k, \text{ if } i \triangleleft j, \\
		&-\sum_{k\in C_j} b_k, \text{ if } j\triangleleft i, 
	\end{aligned}
	\right.\quad \forall\, (i,j)\in E,
	\end{equation*}
where for any node $i$, $C_i$ consists of $i$ and all its children, i.e., $C_i := \{ j \in V \mid j \triangleleft i \} \cup \{i\}$. 
\end{lemma}
\begin{proof}
This result is a simple consequence of the special structure of the node-arc incidence matrix and can be verified directly. 
\end{proof}

Next, we state the blanket assumption on the loss functions $f_i$, $i\in V$, and derive the KKT conditions associated with problem  \eqref{prob:tree_gnio}.
To express our main ideas clearly, we put strong assumptions on $f_i$, such as strong convexity and differentiability. 
However, as can be observed, these strong assumptions could be removed if more subtle analysis is employed.
\begin{assumption}
\label{assum:exist}
Each $f_i:\Re \to \Re$, $i \in V$ in \eqref{prob:tree_gnio} is differentiable and strongly convex.
\end{assumption}

\noindent From the strong convexity of each $f_i$, we know that the objective function in problem \eqref{prob:tree_gnio} is also strongly convex and therefore level-set bounded.
Moreover, by \cite[Theorems 27.1 and 27.2]{rockafellar1970convex}, problem \eqref{prob:tree_gnio} has a unique solution.
We also note that Assumption \ref{assum:exist} holds in some statistical and machine learning problems \cite{best1990active,condat2013direct,tibshirani2011nearly}.
Under Assumption \ref{assum:exist}, we know from  {\cite{rockafellar1970convex}} that each $f_i^*$ is also a strongly convex differentiable function. Moreover, both $f_i'$ and $(f_i^*)'$ are strictly increasing on $\Re$, and for any given $x,y\in \Re$, $y = f_i'(x)$ if and only if $x = (f_i^*)'(y)$.

Now, we are ready to write down the KKT conditions associated with problem \eqref{prob:tree_gnio}.
For $0 \le \lambda, \mu \le +\infty$, let 
\begin{equation*}
\left\{ 
\begin{aligned}
&h^-_\lambda(x) := \delta(x ~|~ x \ge 0), \quad  \mbox{if} \; \lambda = +\infty, \\
&h^-_\lambda(x) := \begin{cases}
-\lambda x,   &  x<0, \\
0,  & x \ge 0,
\end{cases} \quad \mbox{if} \; 0 \le \lambda < +\infty, 
\end{aligned}
\right. 
\quad {\rm and} \quad 
\left\{ 
\begin{aligned}
&h^+_\mu(x) := \delta(x ~|~ x \le 0), \quad  \mbox{if} \; \mu = +\infty, \\
&h^+_\mu(x) := \begin{cases}
0,   &  x\le 0, \\
\mu x,  & x > 0,
\end{cases} \quad \mbox{if} \; 0 \le \mu < +\infty.
\end{aligned}
\right. 
\end{equation*}
For $(i,j) \in E$, we define $h_{i,j}:\Re \to [0, +\infty]$ by
\begin{equation*}
h_{i,j}(x) := h^-_{\lambda_{i,j}}(x) + h^+_{\mu_{i,j}}(x), \quad \forall~ x\in \Re.
\end{equation*}
Clearly, for each $(i,j) \in E$, $h_{i,j}$ is convex and its subdifferential at $x\in \Re$ takes the following form:
\begin{equation}
\label{eq:subhij}
\partial h_{i,j}(x) =  \begin{cases}
 \{ -\lambda_{i,j} \} ,  & {\rm if ~} x < 0, \\
 [-\lambda_{i,j}, \mu_{i,j} ] , & {\rm if ~} x = 0, \\
 \{ \mu_{i,j} \}           ,     & {\rm if ~} x > 0.
\end{cases}   
\end{equation}
Here, $\partial h_{i,j}(x) = \{+\infty \}$ or $\partial h_{i,j}(x) = \{-\infty \}$ should be understood as $\partial h_{i,j}(x) = \emptyset$. We also adopt the conventions in \eqref{eq:subhij} that $[-\infty, +\infty] = (-\infty, +\infty)$, $[-\infty, \alpha] = (-\infty, \alpha]$, and $[\alpha, +\infty] = [\alpha, +\infty)$ for some $\alpha \in \Re$.

Define $H(z) := \sum_{(i,j) \in E} h_{i,j}(z_{i,j})$ for $z \in \Re^{|E|}$, and $F(x) := \sum_{i \in V} f_i(x_i)$ for $x \in \Re^{|V|}$.
{Let $M = -D^T \in \Re^{|E| \times |V|}$, where $D$ is the node-arc incidence matrix associated with $G$.}
{That is, for $e = (i,j) \in E$, $M(e,i) = -1$ and $M(e,j) = 1$ and all other entries of $M$ are zero.}
Let $H_M(x) :=  H(M x)$ for $x \in \Re^{|V|}$.
Then, it can be easily verified that problem \eqref{prob:tree_gnio} can be equivalently rewritten as 
\begin{equation*}
\min_{x\in \Re^{|V|} }~ F(x) + H_M(x).
\end{equation*}
Then, we have the following lemma on the KKT conditions associated with problem \eqref{prob:tree_gnio}.

\begin{lemma}
\label{lem:optimality}
Problem \eqref{prob:tree_gnio} has a unique minimizer $x^* \in \Re^{|V|}$.
Moreover, $x^*$ solves problem \eqref{prob:tree_gnio} if and only if there exists a unique multiplier $z^* \in \Re^{|E|}$, such that $(x^*,z^*)$ satisfies the following KKT system:
\begin{equation}
\label{eq:kkt_tree}
\begin{aligned}
&\sum_{ k:(i,k)\in E} z^*_{i,k} - \sum_{ k:(k,i) \in E} z^*_{k,i} =  f_i'(x^*_i), \quad \forall ~ i \in V, \\
&\; z^*_{i,j} \in \begin{cases}
 \{ -\lambda_{i,j} \} ,  & {\rm if ~}    x^*_{i} > x^*_{j}, \\
 [-\lambda_{i,j}, \mu_{i,j} ] , & {\rm if ~}  x^*_i = x^*_{j} , \\
 \{ \mu_{i,j} \}           ,     & {\rm if ~}  x^*_i < x^*_j, 
\end{cases} \quad  \forall ~ (i,j) \in E.  
\end{aligned}
\end{equation}
\end{lemma}
\begin{proof}
The existence and the uniqueness of the optimal solution to problem \eqref{prob:tree_gnio} follows from the the strong convexity of $F$. Since $F$ is differentiable, we know from \cite[Theorem 23.8]{rockafellar1970convex} that 
\begin{equation*}
 0 \in  F'(x^*) + \partial H_M(x^*).
\end{equation*}
From \cite[Theorem 23.9]{rockafellar1970convex}, it can be seen that $\partial H_M(x^*) = M^T \partial H(Mx^*)$.
Thus, there exists $z^* \in \partial H(Mx^*)$, such that
\begin{equation}
\label{eq:uzstar}
F'(x^*) + M^Tz^* = F'(x^*) - Dz^* = 0.
\end{equation}
Since the $e$-th entry of $Mx^*$ is given by $x^*_{j} - x^*_{i}$, we have from \eqref{eq:subhij} that 
\begin{equation*}
z^*_{i,j} \in \begin{cases}
 \{ -\lambda_{i,j} \} ,  & {\rm if ~}    x^*_{j} - x^*_{i} < 0, \\
 [-\lambda_{i,j}, \mu_{i,j} ] , & {\rm if ~}  x^*_j - x^*_{i} = 0 , \\
 \{ \mu_{i,j} \}           ,     & {\rm if ~}  x^*_j - x^*_i > 0 , 
\end{cases} \quad  \forall ~  (i,j) \in E\red{.}
\end{equation*} 
Thus, we obtain the KKT conditions \eqref{eq:kkt_tree}.
The uniqueness of $z^*$ follows from \eqref{eq:uzstar} and the fact that ${\rm rank}(D) = |E|$.
We thus complete the proof.
\end{proof}

Next, we investigate a naive active set method for solving  problem \eqref{prob:tree_gnio}.
For each edge $(i,j)\in E$, we can associate it with a sign $\# \in \{ <, =, > \}$ to obtain a triple $(i,j,\#)$ representing the relation $x_i \# x_j$. For the consistency, when dealing with edges $(i,j)$ with $\lambda_{i,j} = +\infty$ (or $\mu_{i,j} = +\infty$), the corresponding sign $\#$ can only be chosen from $\{<, =\}$ (or $\{>, =\}$).
We denote by $\cal A$ the collection of all these triples and term it as an {\it active set} associated with problem \eqref{prob:tree_gnio}.
Then, the active set $\cal A$ induces the following $\cal A$-reduced problem from \eqref{prob:tree_gnio}:
\begin{equation}
	\label{prob:tree_gnio_active}
	\begin{array}{cl}
	\min\limits_{x\in\Re^{|V|}} \; & \sum\limits_{i \in V} f_i(x_i) + \sum\limits_{(i,j) \in {\cal A}_{>}} \lambda_{i,j}(x_i - x_{j}) + \sum\limits_{(i,j) \in {\cal A}_{<}} \mu_{i,j} (x_{j} - x_{i}), \\[5mm]
	{\rm s.t. } \; & x_i = x_j, \quad \forall \, (i,j) \in {\cal A}_{=},
	\end{array}
\end{equation}
where ${\cal A}_{\#}:= \{ (i,j) \mid (i,j, \#)\in {\cal A}\}$. 
{If $\mathcal{A}_= = \emptyset$, then \eqref{prob:tree_gnio_active} reduces to an unconstrained optimization problem, which can be efficiently solved since its objective function is separable, smooth and strongly convex}. 
For $i,j \in V$, we say they are \textit{${\cal A}$-connected} if and only if there exists an undirected path in ${\cal A}_{=}$,
{which is obtained by treating all edges in ${\cal A}_{=}$ as undirected edges, that connects $i$ and $j$.
Let $P_{\mathcal{A}}$ be the collection of all ${\cal A}$-connected components of $G$. Then, it is not difficult to observe that $P_{\mathcal{A}}$ is naturally a partition of $G$.
We thus term $P_{\mathcal{A}}$ as the \textit{partition induced by $\mathcal{A}$}.
Without loss of generality, assume $P_{\mathcal{A}} = \{B_k\}_{k=1}^K$ with each $B_k$ being a subtree of $G$, we see that the $\cal A$-reduced problem \eqref{prob:tree_gnio_active} can be decoupled into $K$ independent subproblems as follows:
\begin{equation}
\label{prob:decoupled}
\min_{x\in \Re^{|V_{B_k}|}} \left\{ 
\sum_{i\in V_{B_k}} \hat f_i(x_i)  \mid x_i = x_j, \, \forall \, (i,j)\in E_{B_k}	 \right\}, \quad 1 \le k \le K,
\end{equation}
where for each $i\in V_{B_k}$,
\[
\hat f_i(x_i) := f_i(x_i) + ( \sum_{j:(i,j) \in \mathcal{A}_>} \lambda_{i,j} - \sum_{j:(i,j) \in \mathcal{A}_<} \mu_{i,j} ) x_i + 
( \sum_{l:(l,i) \in \mathcal{A}_<} \mu_{l,i} - \sum_{l:(l,i) \in \mathcal{A}_>} \lambda_{l,i} ) x_i.
\]
Clearly, the simple constraints in problem \eqref{prob:decoupled} can be eliminated. The resulting unconstrained optimization problem has a univariate smooth and strongly convex objective function and thus can be efficiently solved. In this way, we obtain the optimal solution to  the $\cal A$-reduced problem \eqref{prob:tree_gnio_active}.

Unfortunately, there can be up to $3^{|E|}$ different choices for the active set ${\cal A}$. 
Thus, the naive method of exploring all the possible choices of different active sets needs to solve exponential number of  ${\cal A}$-reduced problems. In order to reduce this prohibitive computational costs, we introduce a novel active-set based recursive algorithm in the next section.

\section{An recursive algorithm for solving problem $(\ref{prob:tree_gnio})$}
\label{sec:algorithm}
In this section, we present our recursive algorithm for solving problem \eqref{prob:tree_gnio}. 
We first claim that, without loss of generality, the directed tree $G$ in \eqref{prob:tree_gnio} can be assumed to be an arborescence with the node $1$ to be its root. 
Moreover, we can decompose $G$ into a sequence of subtrees $\{G_m = (V_m, E_m)\}_{m=1}^n$, where $G_1 \subset G_2 \subset \cdots \subset G_n = G$ and $V_m = \{1, 2, \ldots, m \}$ for $1 \le m \le n$, and the set of edges $E_{m+1} \setminus E_m$ contains exactly one edge $(i_m , m+1)$, where $i_m \in V_m$.
Further details are deferred to the Appendix.

For each $1\le m\le n$, problem \eqref{prob:tree_gnio}, when restricted to the the subtree $G_m$, takes the following form: 
\begin{equation}
\label{prob:gnio_gm}
\min_{x\in\Re^{|V_{m}|}} \; \sum_{i \in V_{m} } f_i(x_i) + \sum_{(i,j) \in E_m} \lambda_{{i,j}} (x_i - x_{j})_{+} + \sum_{(i,j) \in E_m} \mu_{i,j} (x_{j} - x_{i})_{+}.
\end{equation}
From Lemma \ref{lem:optimality}, it is not difficult to see that the unique primal-dual optimal pair to problem \eqref{prob:gnio_gm}, denote by $(x^{(m)}, z^{(m)}) \in \Re^{|V_m|} \times \Re^{|E_m|}$, satisfies the following KKT system:
\begin{equation}
\label{eq:reduced}
\begin{aligned}
&\sum_{ k:(i,k)\in E_m} z_{i,k} - \sum_{ k:(k,i) \in E_m} z_{k,i} = f_i'(x_i), \quad \forall ~ i \in V_m, \\
& z_{i,j} \in \begin{cases}
 \{ -\lambda_{i,j} \} ,  & {\rm if ~}  x_{i} > x_{j}, \\
 [-\lambda_{i,j}, \mu_{i,j} ] , & {\rm if ~} x_i = x_{j}, \\
 \{ \mu_{i,j} \}           ,     & {\rm if ~}  x_i < x_{j}, 
\end{cases} \quad  \forall ~ (i,j) \in E_m. 
\end{aligned}
\end{equation}
The unique optimal pair $(x^{(m)},z^{(m)})$ is also referred to as the $G_m$-\textit{optimal pair} for convenience.
By carefully exploiting the special structures in the KKT conditions \eqref{eq:reduced}, we propose to solve problem \eqref{prob:tree_gnio}
in a recursive fashion. Specifically, we will recursively generate the $G_{m+1}$-optimal pair $(x^{m+1},z^{(m+1)})$ from the $G_{m}$-optimal pair $(x^{(m)},z^{(m)})$ for $m=1,\ldots, n-1$.

We summarize the detailed steps of {the above recursive approach} in Algorithm \ref{alg:outer}. 
In the algorithm, the \textit{generate} subroutine is designed to generate the $G_{m+1}$-optimal pair from the $G_m$-optimal pair. 
In the next subsection, we will show that this procedure is accomplished via a novel active-set searching scheme. Hence, it is natural for us to call Algorithm \ref{alg:outer} an active-set based recursive approach (ASRA).

\begin{algorithm}[htbp]
	\caption{ ASRA: An active-set based recursive approach for solving problem $(\ref{prob:tree_gnio})$  }
	\label{alg:outer}
	\begin{algorithmic}[1]
		\State {{\bf Initialize:} $x^{(1)}_1 = (f_1^*)'(0) \in \Re$, and $z^{(1)} = \emptyset$} 
		\For {$m = 1,\ldots,n-1$}
		\State { $(x^{(m+1)} ,z^{(m+1)} ) = {\it generate}( x^{(m)}, z^{(m)} , G_{m+1}  )$  }	
		\EndFor
       \State {{\bf Return:} $(x^{(n)},z^{(n)}) \in \Re^{n} \times \Re^{n-1}$ }
	\end{algorithmic}
\end{algorithm}

\subsection{The \textit{generate} subroutine}
To efficiently obtain the $G_{m+1}$-optimal pair {from the given $G_m$-optimal pair}, we shall investigated the KKT conditions associated with the subproblem induced by the subtree $G_{m+1}$. Specially, it takes the following form:
\begin{align}
\label{align:vm}
&\sum_{ k:(i,k)\in E_{m}} z_{i,k} - \sum_{ k:(k,i) \in E_{m}} z_{k,i} = f_i'(x_i), \quad \forall ~ i \in V_m \backslash \{ i_m \},\\
\label{align:em}
& z_{i,j} \in \begin{cases}
 \{ -\lambda_{i,j} \} ,  & {\rm if ~} x_{i} > x_{j}, \\
 [-\lambda_{i,j}, \mu_{i,j} ] , & {\rm if ~} x_i =x_{j}, \\
 \{ \mu_{i,j} \}           ,     & {\rm if ~} x_i <x_{j}, 
\end{cases} \quad  \forall ~ (i,j) \in E_m, \\
\label{align:im}
&\sum_{ k:(i_m,k) \in E_{m} } z_{i_m,k} -  \sum_{ k:(k,i_m) \in E_{m}} z_{k,i_m} +  z_{i_m, m+1} = f'_{i_m}(x_{i_m}), \\
\label{align:m+1}
&- z_{i_m, m+1} = f'_{m+1}(x_{m+1}), \\
\label{align:em+1}
& z_{i_m,m+1 } \in \begin{cases}
 \{ -\lambda_{i_m,m+1 } \} ,  & {\rm if ~} x_{i_m} > x_{m+1}, \\
 [-\lambda_{i_m,m+1 }, \mu_{i,j} ] , & {\rm if ~} x_{i_m} = x_{m+1}, \\
 \{ \mu_{i_m,m+1 } \}           ,     & {\rm if ~}  x_{i_m} < x_{m+1}. 
\end{cases}
\end{align}
As one can observe, instead of writing the KKT conditions as a whole set of equations, we have singled out those, namely \eqref{align:im}, \eqref{align:m+1} and \eqref{align:em+1}, associated with the dual variable $z_{i_m, m+1}$, which corresponds to the newly added edge $ \{(i_m, m+1)\} = E_{m+1}\setminus E_{m}$. 
Based on the above KKT conditions, we have the following proposition regarding the sign of $z_{i_m, m+1}$.
\begin{proposition}
\label{prop:direction}
It holds that $z^{(m+1)}_{i_m, m+1}  f'_{m+1}( x^{(m)}_{i_m} ) \le 0$, where  $(x^{(m)}, z^{(m)})$  and $(x^{(m+1)},z^{(m+1)})$ are the $G_m$-optimal pair and the $G_{m+1}$-optimal pair, respectively.  
\end{proposition}
\begin{proof}
Note that when $f'_{m+1}( x^{(m)}_{i_m} ) = 0$, the desired result naturally holds.
For the remaining parts, we only prove the case where $f'_{m+1}( x^{(m)}_{i_m} ) > 0$, since the proof for the case with $f'_{m+1}( x^{(m)}_{i_m} ) < 0$ can be easily modified from the arguments here.

Suppose that $f'_{m+1}( x^{(m)}_{i_m} ) > 0$, then we shall prove that $z^{(m+1)}_{i_m, m+1} \le 0$. 
Assume on the contrary that $z^{(m+1)}_{i_m, m+1} > 0$. Then, from \eqref{align:m+1}, we have $x^{(m+1)}_{m+1} = (f_{m+1}^{*})'(- z^{(m+1)}_{i_m,m+1}) < (f_{m+1}^{*})'(0)$.
Moreover, \eqref{align:em+1} implies that $x^{(m+1)}_{m+1} \ge x^{(m+1)}_{i_m}$. 
Thus, we have from the strict monotonicity of $(f_{m+1}^*)'$ the following inequality:
\begin{equation}
\label{eq:im_relations}
x^{(m)}_{i_m} > (f_{m+1}^{*})'(0) > {x}^{(m+1)}_{m+1} \ge {x}^{(m+1)}_{i_m}.
\end{equation}

Now, from \eqref{align:vm}, \eqref{align:em}, and \eqref{align:im}, we see that $\widetilde{x} \in \Re^{|V_m|}$ with $\widetilde{x}_i = {x}^{(m+1)}_i $ for $i\in V_m$ is the optimal solution to the following optimization problem:
\begin{equation*}
\min_{x\in\Re^{|V_m|}} \; F_1(x) := \sum_{i \in V_m} f_i(x_i) + \sum_{(i,j) \in E_m} \left\{ \lambda_{{i,j}} (x_i - x_{j})_{+} + \mu_{i,j} (x_{j} - x_{i})_{+} \right\} - z^{(m+1)}_{i_m, m+1} x_{i_m}.
\end{equation*}
Meanwhile, since $(x^{(m)},z^{(m)})$ is the $G_m$-optimal pair, $x^{(m)}$ is the optimal solution to the following optimization problem:
\begin{equation*}
\min_{x\in\Re^{|V_m|}} \; F_0(x) := \sum_{i \in V_m} f_i(x_i) + \sum_{(i,j) \in E_m}  \left\{ \lambda_{{i,j}} (x_i - x_{j})_{+} + \mu_{i,j} (x_{j} - x_{i})_{+} \right\} .
\end{equation*}
Then, it holds that 
\begin{equation*}
0 \ge{} F_1(\widetilde{x}) - F_1(x^{(m)}) 
= {}  F_0(\widetilde{x}) - F_0(x^{(m)}) + {z}^{(m+1)}_{i_m,m+1}( x^{(m)}_{i_m} - \widetilde{x}_{i_m}).
\end{equation*}
Since $F_0(\widetilde{x}) - F_0(x^{(m)}) \ge 0$, ${z}^{(m+1)}_{i_m,m+1}>0$, and $ \widetilde{x}_{i_m}  = {x}^{(m+1)}_{i_m}$, we have $ x^{(m)}_{i_m}  - {x}^{(m+1)}_{i_m}\le 0$, which contradicts to \eqref{eq:im_relations}.
Thus, we have $z^{(m+1)}_{i_m, m+1} \le 0$ and $z^{(m+1)}_{i_m, m+1} f'_{m+1}(x^{(m)}_{i_m}) \le 0$,  and complete the proof.
\end{proof}

From Proposition \ref{prop:direction}, we can determine the sign of $z^{(m+1)}_{i_m,m+1}$ by the value of $f'_{m+1}(x_{i_m}^{(m)})$. 
Moreover, if $f'_{m+1}(x_{i_m}^{(m)}) = 0$, we can easily construct the $G_{m+1}$-optimal pair as follows:
\begin{equation*}
x_{i}^{(m+1)} = \left\{ 
\begin{aligned}
	& x_{i}^{(m)}, \, \forall\, i\in V_m, \\
	& { x_{i_m}^{(m)}} , \, i = m+1,
\end{aligned}	\right.
 \mbox{ and }\, 
z_{i,j}^{(m+1)} = 
	\left\{ 
	\begin{aligned}
	& z_{i,j}^{(m)}, \, \forall\, (i,j)\in E_{m}, \\
	& 0, \, (i,j) = (i_m, m+1).
	\end{aligned}
	\right.
\end{equation*}
Hence, we focus on the case with $f'_{m+1}(x_{i_m}^{(m)}) \neq 0$ in the subsequent discussions.  For this purpose, we consider the following parametric optimization problem with the parameter $t \in \Re$:
\begin{equation}
\label{prob:gm}
\min_{x\in\Re^{|V_{m+1}|}} \sum_{i \in V_{m+1} } f_i(x_i) + \sum_{(i,j) \in E_m} \{ \lambda_{{i,j}} (x_i - x_{j})_{+} + \mu_{i,j} (x_{j} - x_{i})_{+} \} - t( x_{i_m} - x_{m+1} ),
\end{equation}
whose KKT conditions are presented below:
\begin{equation}
\begin{aligned}
\label{eq:kkt_gm}
&\sum_{ k:(i,k)\in E_{m}} z_{i,k} - \sum_{ k:(k,i) \in E_{m}} z_{k,i} + { 1_{ \{i = i_m\}} } t = f_i'(x_i), \quad \forall ~ i \in V_m,\\
& z_{i,j} \in \begin{cases}
 \{ -\lambda_{i,j} \} ,  & {\rm if ~}  x_{i} > x_{j}, \\
 [-\lambda_{i,j}, \mu_{i,j} ] , & {\rm if ~}x_i = x_j, \\
 \{ \mu_{i,j} \}           ,     & {\rm if ~}  x_i < x_j, 
\end{cases} \quad  \forall ~ (i,j) \in E_m, \\
&- t = f'_{m+1}(x_{m+1}). \\
\end{aligned}
\end{equation}
Since each $f_i$ is strongly convex, problem \eqref{prob:gm} has a unique optimal solution, denoted by $x^*(t)$,  for each $t\in \Re$. 
Moreover, using the Fenchel-Rockafellar duality theorem \cite{rockafellar1970convex} and the differentiability of each $f_i$, we know that there exists a unique dual optimal solution to problem \eqref{prob:gm}, denoted by $z^*(t)$, which together with $x^*(t)$ satisfies the KKT conditions \eqref{eq:kkt_gm}.
If for certain $t^* \in \Re$, it holds that
\begin{equation}
\label{eq:t_star}
t^* \in  \begin{cases}
 \{ -\lambda_{i_m,m+1 } \} ,  & {\rm if ~}  x^*_{i_m}(t^*) > x^*_{m+1}(t^*), \\
 [-\lambda_{i_m,m+1 }, \mu_{i_m,m+1} ] , & {\rm if ~} x^*_{i_m}(t^*) = x^*_{m+1}(t^*), \\
 \{ \mu_{i_m,m+1 } \}           ,     & {\rm if ~} x^*_{i_m}(t^*) <x^*_{m+1}(t^*).
\end{cases}
\end{equation}
Then, by comparing the equations \eqref{eq:kkt_gm} and \eqref{eq:t_star} and the KKT conditions in equations \eqref{align:vm} to \eqref{align:em+1}, we can obtain the $G_{m+1}$-optimal pair based on $(x^*(t^*),z^*(t^*))$. 
Indeed, the $G_{m+1}$-optimal pair $(x^{(m+1)}, z^{(m+1)})$ can be constructed via \[x^{(m+1)} = x^*(t^*), \mbox{ and }  z^{(m+1)}_{i,j} = z^*_{i,j}(t^*) \mbox{ for } (i,j) \in E_m \red{,} \mbox { and } z^{(m+1)}_{i_m,m+1} = t^*.\]
This observation also indicates that one can determine the sign of $t^*$ using Proposition \ref{prop:direction}.

To find the desired $t^*$, we start from the initial guess $t_0 = 0$. 
We note that when $t_0 =0$, the corresponding primal-dual optimal pair $( x^*(t_0), z^*(t_0) )$ is readily known with $x_i^*(t_0) = x_i^{(m)}$ for $i \in V_m$ and $x_{m+1}^*(t_0) = (f_{m+1}^*)'(-t_0)$, and $z^*(t_0) = z^{(m)}$. 
Then, we can easily check if $t_0=0$ satisfies \eqref{eq:t_star} by comparing $x_{m+1}^*(t_0)$ and $x_{i_m}^*(t_0)$. 
If $x_{m+1}^*(t_0) \neq x_{i_m}^*(t_0)$, we can use Proposition \ref{prop:direction} to determine if $t$ should be decreased or increased. 
Assume without loss of the generality that $f'_{m+1}(x_{i_m}^*(t_0)) = f'_{m+1}(x_{i_m}^{(m)}) > 0$.  
From the above discussions and Proposition \ref{prop:direction}, we see that $t^* < 0$. 
Then, we rely on an active-set strategy to iteratively update our guess of $t^*$.

Starting from the initial guess $t_0 = 0$, we denote the active set corresponding to $E_m$ in \eqref{prob:gm} by 
\begin{equation}
	\label{eq:A0}
	\mathcal{A}^0 = \{ (i,j, \#) \mid (i,j) \in E_m, \; x_i^*(t_0) \, \# \, x_j^*(t_0) \}, \, \mbox{ where } \# \in \{<, = , > \}.
\end{equation}
Then, we add the equality constraints induced by edges in $\mathcal{A}^0_{=}$ to problem \eqref{prob:gm} and obtain the $\mathcal{A}^0$-{reduced problem} of problem \eqref{prob:gm}.
The key observation is that the primal-dual optimal solution pair to the $\mathcal{A}^0$-reduced problem can be written in a semi-closed form as functions of the parameter $t$, denoted by $(x^{0}(t), z^{0}(t))$.
Then, we construct a dual candidate $\tilde{z}^{0}(t)$ to problem \eqref{prob:gm} as follows:
\begin{equation*}
	\tilde{z}_{i,j}^{0}(t) = \left\{ 
	\begin{aligned}
	& z_{i,j}^{0}(t), \, \mbox{ if }(i,j)\in \mathcal{A}^0_{=}, \\[2pt]
	& z_{i,j}^*(t_0), \, \mbox{ otherwise},
	\end{aligned}
	\right.\quad \forall \, (i,j)\in E_m.
\end{equation*}
We will show that if $(x^{0}(t), \tilde{z}^{0}(t))$ satisfies the complementarity conditions in \eqref{eq:kkt_gm}, i.e., 
\[
	\tilde{z}^{0}_{i,j}(t) \in \begin{cases}
		\{ -\lambda_{i,j} \} ,  & {\rm if ~} x^{0}_{i}(t) > x^{0}_{j}(t), \\
		[-\lambda_{i,j}, \mu_{i,j} ] , & {\rm if ~}  x^{0}_{i}(t) = x^{0}_j(t), \\
		\{ \mu_{i,j} \}           ,     & {\rm if ~}  x_i^{0}(t)< x_j^{0}(t), 
	   \end{cases} \quad  \forall ~ (i,j) \in E_m,
\]
then $(x^{0}(t), \tilde z^{0}(t))$ is the primal-dual optimal solution pair to problem \eqref{prob:gm}.

Based on this observation, a new guess of $t^*$ is constructed by searching for the smallest possible $t_1$ such that $-\lambda_{i_m,m+1} \le t^* \le t_1 \le t_0 = 0$ and  $( x^{0}(t_1), \tilde z^{0}(t_1))$ still satisfies the above complementarity conditions. 
Then, we have $(x^*(t_1), z^*(t_1)) = ( x^{0}(t_1), \tilde z^{0}(t_1))$ and we can check if $t_1$ satisfies the system \eqref{eq:t_star}. 
If not, then a new active set  ${\cal A}^1$ is constructed and the above process continues until $t^*$ is found. 
In a nutshell, our approach is summarized in the following flowchart:
\begin{equation*}
(t_0 , x^*(t_0), z^*(t_0), \mathcal{A}^0 )\Rightarrow \cdots \Rightarrow (t_q, x^*(t_q), z^*(t_q), \mathcal{A}^q )  \Rightarrow \cdots \Rightarrow (t^*, x^*(t^*), z^*(t^*), \mathcal{A}^*).
\end{equation*}
In what follows, we shall discuss the detailed steps of our procedure and we will prove that the search process of $t^*$ terminates in at most $2m-1$ steps.

At $t_q$ with $t^* < t_q\le t_0$, we assume that {$(x^*(t_q)$, $z^*(t_q))$}, and the corresponding active set ${\cal A}^q$ are available. Then,
we construct the following $\mathcal{A}^q$-reduced parametric optimization problem with parameter $t\in \Re$:
\begin{equation}\label{prob:gm_active}
	\begin{array}{rl}
		\displaystyle \min_{x\in\Re^{|V_{m+1}|}} \; &  \displaystyle  \sum_{i \in V_{m+1}} f_i(x_i) + \displaystyle  \sum_{(i,j) \in \mathcal{A}^{q}_{>}} \lambda_{{i,j}} (x_i - x_{j}) + \displaystyle  \sum_{(i,j) \in \mathcal{A}^{q}_{<}} \mu_{i,j} (x_{j} - x_{i}) - t (x_{i_m} - x_{m+1}), \\[7mm]
		{\rm s.t. } \;  & x_i = x_j, \quad \forall \, (i,j) \in \mathcal{A}^{q}_{=},	
	\end{array}
\end{equation}
whose unique primal-dual optimal pair is denoted by $(x^{q}(t), z^{q}(t))$. 
{If $\mathcal{A}^{q}_{=} = \emptyset$, then we set $z^q(t) = \emptyset$.}
Here, we require the following compatibility conditions between ${\cal A}^q$ and $(x^*(t_q),z^*(t_q))$, which also servers as an induction hypothesis.
\begin{assumption}
\label{assump:tq}
The active set ${\cal A}^q$ and the primal-dual pair $(x^*(t_q),z^*(t_q))$ are compatible. 
That is, $x^*(t_q)$ is the optimal solution to the problem \eqref{prob:gm_active} at $t = t_q$, i.e., $x^q(t_q) = x^*(t_q)$ and the corresponding dual optimal solution $z^q(t_q)$ can be constructed via $z^q_{i,j}(t_q) = z_{i,j}^*(t_q)$ for $(i,j) \in \mathcal{A}^q_{=}$. 
Moreover, it holds that $x_{i_m}^*(t_q) - x^*_{m+1}(t_q) > 0$.
\end{assumption}
\noindent We shall emphasize that according to the construction of ${\cal A}^0$, it is not difficult to observe that the active set ${\cal A}^0$ and the primal-dual pair $(x^*(t_0),z^*(t_0))$ are compatible, and $x_{i_m}^*(t_0) - x^*_{m+1}(t_0) > 0$.
Next, we focus on obtaining $(t_{q+1}, x^*(t_{q+1}), z^*(t_{q+1}), \mathcal{A}^{q+1})$ from  $(t_{q}, x^*(t_{q}), z^*(t_{q}), \mathcal{A}^q )$.

We start by investigating the optimal primal-dual solution pair corresponding to problem \eqref{prob:gm_active}. Particularly, instead of solving problem \eqref{prob:gm_active} for each $t\neq t_q$, we derive in the following proposition the semi-closed formulas for  $(x^{q}(t), z^{q}(t))$ under Assumption \ref{assump:tq}. 
We also show that the optimal primal-dual solution pair of problem \eqref{prob:gm} can be obtained from  $(x^{q}(t), z^{q}(t))$ provided that some complementarity  conditions hold.

\begin{proposition}
\label{prop:closedform} 
Let $P_{ \mathcal{A}^q }$ be the partition of $G_m$ induced by $\mathcal{A}^q$ and $B^q \in P_{ \mathcal{A}^q  }$ be the subtree such that $i_m\in B^q$. 
Then, under Assumption \ref{assump:tq}, for any $t \in \Re$, the primal optimal solution $x^{q}(t)$ takes the following form:
\begin{equation}
\label{eq:update_x}
\left\{ 
\begin{aligned}
&x^{q}_i(t)	= x_i^*(t_q), \quad \forall\, i \in V_m \backslash V_{B^q}, \\[2mm]
&x^{q}_i (t) = \big((\sum_{i \in V_{B^q}} f_i)^{*}\big)' \bigg( t + \beta^{q} \bigg), \quad \forall ~ i \in V_{B^q}, \\[2mm]
&x^{q}_{m+1}(t)  = ( f^*_{m+1})' ( -t ), \\ 
\end{aligned}
\right.
\end{equation}
where
\[
\beta^q = \sum_{(i,k) \in E_m\atop i\in V_{B^q}, k \notin V_{B^q}} z^*_{i,k}(t_q) - \sum_{(k,i) \in E_m \atop k\notin V_{B^q}, i \in V_{B^q}} z^*_{k,i}(t_q).
\]
Pick $i_m$ as the ancestor of $B^q$. Then, for any $t\in \Re$, $z^{q}(t)$ is given by 
\begin{equation}
\label{eq:update_z}
\left\{ 
\begin{aligned}
&z_{i,j}^{q}(t) = z_{i,j}^*(t_q), \quad \forall (i,j)\in {\cal A}^q_{=} \backslash E_{B^q}, \\[2mm]
&z_{i,j}^{q}(t) = 
\left\{ 
\begin{aligned}
&\sum_{l \in C_i} f'_l(x_l^{q}(t))  - \alpha^q_{i,j} , \, \mbox{ if } i \triangleleft j, \\[2mm]
		& \sum_{l \in C_j} - f'_l(x_l^{q}(t))  + \alpha^q_{i,j}, \, \mbox{ if } i \triangleleft j, 
	\end{aligned}
\quad \forall (i,j) \in E_{B^q},
	\right.
	\end{aligned}
	\right.
\end{equation}
where $C_i := \{j\in V_{B^q}\mid j \triangleleft i \} \cup \{ i \}$ for any $i\in V_{B^q}$, and 
\[
\alpha^q_{i,j} =\left\{ 
\begin{aligned}
& \sum_{(l,k)\in E_m \atop l\in C_i, k\notin V_{B^q}} z_{l,k}^*(t_q) - \sum_{(k,l)\in E_m \atop l\in C_i, k\notin V_{B^q}} z_{k,l}^*(t_q) , \, \mbox{ if } i \triangleleft j, \\[2mm]
& \sum_{(l,k)\in E_m \atop l\in C_j, k\notin V_{B^q}} z_{l,k}^*(t_q) - \sum_{(k,l)\in E_m \atop l\in C_j, k\notin V_{B^q}} z_{k,l}^*(t_q) , \, \mbox{ if } j \triangleleft i, 
	\end{aligned}
\quad \forall (i,j) \in E_{B^q}.
	\right.
\]

Let $\Omega^q = \{ (i,j) \in E_m \backslash E_{B^q} \mid \mbox{exactly one of } i \mbox{ and } j \mbox{ is in } V_{B^q}\}$.
If
\begin{align}
\label{eq:cs_EB}
& z_{i,j}^{q}(t) \in [-\lambda_{i,j}, \mu_{i,j}], \quad \forall\, (i,j)\in E_{B^q}, \\
\label{eq:zstar_xatq}
& z^*_{i,j}(t_q) \in \begin{cases}
	\{ -\lambda_{i,j} \} ,  & {\rm if ~} x^{q}_{i}(t) > x^{q}_{j}(t), \\
	[-\lambda_{i,j}, \mu_{i,j} ] , & {\rm if ~}  x^{q}_{i}(t) = x^{q}_j(t), \\
	\{ \mu_{i,j} \}           ,     & {\rm if ~}  x_i^{q}(t)< x_j^{q}(t), 
\end{cases} 
\quad \forall \,(i,j)\in \Omega^q,
\end{align}
then $(x^{q}(t), \tilde z^{q}(t))$ solves the KKT system \eqref{eq:kkt_gm}, where \begin{equation}
	\label{eq:tildez_1}
	\tilde{z}_{i,j}^{q}(t) = \left\{ 
	\begin{aligned}
		& z_{i,j}^{q}(t), \, \mbox{ if }(i,j)\in {\cal A}^q_=, \\[2pt]
		& z_{i,j}^*(t_q), \, \mbox{ otherwise},
	\end{aligned}
	\right.\quad \forall \, (i,j)\in E_m.
\end{equation}  
\end{proposition}

\begin{proof}
Without loss of generality, we can assume that  $P_{ \mathcal{A}^q } = \{B_k\}_{k=1}^K \cup B^{q}$ where $B_k$, $1\le k \le K$, and $B^q$ are subtrees of $G_{m}$. Then, problem \eqref{prob:gm_active} can be decomposed into $K+2$ independent subproblems on each subtree $B_k$ and $B^q$ and the singleton $\{ m+1\}$.
 Note that the parameter $t$ only appears in the subproblems corresponding to the subtree $B^{q}$ and the singleton $\{ m+1\}$. Hence, from Assumption \ref{assump:tq}, it is not difficult to deduce that for any $t \in \Re$,
\[
x^{q}_i(t)	= x_i^*(t_q), \, i \in V_m \backslash V_{B^q}, \quad \mbox{ and } \quad  z^{q}_{i,j}(t) = z_{i,j}^*(t_q), \, (i,j)\in {\cal A}^q_=\setminus E_{B^q}.
\]
The subproblem associated with $\{ m+1 \}$ is easily solved via $x^q_{m+1}(t) = (f_{m+1}^*)'(-t)$.
Therefore, we only need to focus on the subproblem associated with the subtree $B^q$:
\begin{equation}
	\label{prob:subproblem_B}
	\min_{x\in \Re^{|V_{B^q}|}} \left\{ 
	\sum_{i\in V_{B^q}} \hat f_i(x_i) - t x_{i_m} \mid x_i = x_j, \, \forall \, (i,j)\in E_{B^q}	
	\right\},
\end{equation}
where 
\[
\hat f_i(x_i): = f_i(x_i) + \sum_{k \not\in V_{B^q} \atop (k,i)\in E_m } z_{k,i}^*(t_q) x_i  - \sum_{ k \not\in V_{B^q} \atop (i,k)\in E_m} z_{i, k}^*(t_q) x_i, \quad  \forall \, i \in V_{B^{q}}.
\]
Let $\mathcal{L}$ be the Lagrangian function associated with problem \eqref{prob:subproblem_B}
\[
{\mathcal{L}(x;z)} = \sum_{i\in V_{B^q}} \hat f_i(x_i) - t x_{i_m} - \sum_{(i,j)\in E_{B^q}} z_{i,j}(x_i - x_j), \quad \forall\, (x,z)\in \Re^{|V_{B^q}|}\times \Re^{|E_{B^q}|}.
\]
Then, the optimal primal-dual solution pair to problem \eqref{prob:subproblem_B} satisfies the following KKT system:
\begin{equation}
\label{eq:kkt_subproblem}
\left\{
\begin{aligned}
& x_i = x_j, \quad \forall\, (i,j)\in E_{B^q}, \\
& {f'_i}(x_i) + \sum_{ k \not\in V_{B^q} \atop (k,i)\in E_m} z_{k,i}^*(t_q) + \sum_{ k \in V_{B^q} \atop (k,i) \in E_{B^q}} z_{k,i}  - \sum_{  k \not\in V_{B^q} \atop (i, k) \in E_m} z_{i, k}^*(t_q) 
	- \sum_{k\in V_{B^q} \atop (i,k) \in E_{B^q}} z_{i,k} - {1_{ \{i=i_m\} }} t = 0, 
\quad \forall \, i\in V_{B^q}.
\end{aligned}
\right.
\end{equation}
Summing over all $i\in V_{B^q}$, we deduce from the above system that 
\[
\sum_{i\in V_{B^q}} {f'_i}(x^{q}_i (t)) = -\sum_{(k,i)\in E_m \atop k \not\in V_{B^q}, i\in V_{B^q}} z_{k,i}^*(t_q) + \sum_{  (i, k) \in E_m \atop  i\in  V_{B^q}, k \not\in V_{B^q}} z_{i, k}^*(t_q) + t, 
\]
i.e., 
\[
x^{q}_i (t) = ((\sum_{i \in V_{B^q}} f_i)^{*})' ( t + \sum_{(i,k) \in E_m\atop i\in V_{B^q}, k \notin V_{B^q}} z^*_{i,k}(t_q) - \sum_{(k,i) \in E_m \atop k\notin V_{B^q}, i \in V_{B^q}} z^*_{k,i}(t_q) ), \quad \forall ~ i \in V_{B^q}.
\]

{Next}, we obtain from the above KKT system {\eqref{eq:kkt_subproblem}} the following linear system corresponding to  $z_{i,j}$ for $(i,j)\in E_{B^q}$:
\[
\sum_{k:(i,k)\in E_{B^q}} z_{i,k} - \sum_{k:(k,i)\in E_{B^q}} z_{k,i} =  {f'_i}( x_i^q(t) ) + \sum_{k\notin V_{B^q}\atop (k,i)\in E_m} z^*_{k,i}(t_q) - \sum_{k\notin V_{B^q}\atop (i,k)\in E_m} z^*_{i,k}(t_q), \quad \forall \, i\in V_m \setminus \{i_m\}.
\]
Since $i_m$ is the ancestor of the subtree $B$, we obtain from Lemma \ref{lem:na_sol} the updated formula for  $z_{i,j}^{q}(t)$, $(i,j)\in E_{B^q}$. Thus, we proved \eqref{eq:update_z}. 

Finally, it is not difficult to see that if {the assumed conditions \eqref{eq:cs_EB} and \eqref{eq:zstar_xatq} are satisfied}, then $x^{q}(t)$ and $\tilde z^{q}(t)$ satisfy the complementarity conditions in the KKT system \eqref{eq:kkt_gm}.
The rest equations in \eqref{eq:kkt_gm} hold automatically by noting \eqref{eq:tildez_1} and the KKT system \eqref{eq:kkt_subproblem}.
\end{proof}

Using the semi-closed formulas in Proposition \ref{prop:closedform}, we compute the following lower bound $\Delta t_q \le 0$:
\[
\Delta t_q := \min \left\{ \Delta t \mid  \mbox{ 
\eqref{eq:cs_EB} and \eqref{eq:zstar_xatq} hold for all } t\in [t_q + \Delta t, t_q]
\right\}.
\]
The computations are divided into two parts. 
Firstly, we focus on the value of $z_{i,j}^{q}(t)$ for $(i,j)\in E_{B^{q}}$.
For any $(i,j)\in E_{B^{q}}$, we note that $z_{i,j}^{q}(t_q)\in [-\lambda_{i,j}, \mu_{i,j}]$ and $z_{i,j}^{q}(t)$ is increasing if $i \triangleleft j$ and is decreasing if $j\triangleleft i$ with respect to $t$ from \eqref{eq:update_z}. 
We define the threshold $\Delta (E_{B^{q}})$ as follows:
\begin{equation}
\label{def:delta_EB}
	\Delta (E_{B^{q}}):= \left\{
\begin{aligned}
	& \max_{(i,j)\in E_{B^{q}}}  \Delta t_{i,j}, \quad \mbox{if } E_{B^{q}}\neq \emptyset, \\[2pt]
   & -\infty, \quad \mbox{otherwise}.
\end{aligned}
	\right.
\end{equation}
Here, each {$\Delta t_{i,j} \le 0$} solves
\begin{equation}
\label{eq:delta_inEB}
z_{i,j}^{q}(t_q + \Delta t_{i,j}) = {-\lambda_{i,j}} , \quad \mbox{if } \, i\triangleleft j, \quad \mbox{and} \quad  z_{i,j}^{q}(t_q + \Delta t_{i,j}) = {\mu_{i,j}} , \quad \mbox{if } \, j \triangleleft i.
\end{equation}
Next, the relations in \eqref{eq:zstar_xatq} corresponding to the edges in $\Omega^q$ are examined. For this purpose,  we divide $\Omega^q$ into two parts, namely,
\begin{equation}
\label{eq:Omegaq}
\Omega^q_+ = \left\{ (i,j) \in \Omega^q \mid i \in V_{B^{q}}, \, j\in V_m \backslash V_{B^{q}}
\right\}	\mbox{ and } \Omega^q_- = \left\{ (i,j) \in \Omega^q \mid i \in V_m \backslash V_{B^{q}}, \, j\in V_{B^{q}}
\right\},
\end{equation}
and handle them separately. From \eqref{eq:update_x}, we know that $x_i^{q}(t)$ takes the same value for all $i\in V_{B^{q}}$ and is increasing with respect to $t$. Hence, we can simply denote $x_{B^q}(t) = x_i^q(t)$ for any $i\in V_{B^q}$.
Then, we compute the threshold $\Delta(\Omega^{q}) := \max\{ \Delta(\Omega^q_+), \Delta(\Omega^q_-)\}$, where 
\begin{equation}
\label{def:delta_O+}
	\Delta(\Omega^q_+) := \left\{
	\begin{aligned}
	& \Delta \overline{t} \mbox{ satisfying } 
	x_{B^q}(t_q + \Delta \overline{t}) = \max_{(i,j)\in \Omega^q_+ \cap {\cal A}^q_>} x^*_j(t_q), \quad  \mbox{if } \; \Omega^q_+ \cap {\cal A}^q_> \neq \emptyset, \\[2pt]
	& -\infty, \mbox{ otherwise,}
	\end{aligned}
	\right.
\end{equation}
and 
\begin{equation}
\label{def:delta_O-}
	\Delta(\Omega^q_-):= \left\{
	\begin{aligned}
	& \Delta \overline{t} \mbox{ satisfying } 
	x_{B^q}(t_q + \Delta \overline{t} ) = \max_{(i,j) \in \Omega^q_-  \cap {\cal A}^q_<} x^*_i(t_q), \quad  \mbox{if} \; \Omega^q_- \cap {\cal A}^q_< \neq \emptyset, \\[2pt]
	& -\infty, \mbox{ otherwise.}
	\end{aligned}
	\right.
\end{equation}
It can be easily verified that 
\begin{equation}
	\label{eq:dtq}
	\Delta t_q = \max\{\Delta(E_{B^{q}}), \Delta(\Omega^q)\}.
\end{equation}
Thus, using Proposition \ref{prop:closedform}, we can obtain the semi-closed form for the optimal solution $x^*(t)$, as well as its corresponding dual optimal solution $z^*(t)$, to problem \eqref{prob:gm} for any $t\in [t_q + \Delta t_q , t_q]$.

Now, we are ready to discuss the search of $t_{q+1}$.
Note that according to Assumption \ref{assump:tq}, we have 
\[x^q_{i_m}(t_q) - x^q_{m+1}(t_q) = x^*_{i_m}(t_q) - x^*_{m+1}(t_q)  > 0.\]
Using the closed-form formulas in Proposition \ref{prop:closedform}, we know that $x^q_{i_m}(t) - x^q_{m+1}(t)$ is strictly increasing with respect to $t$, and we can obtain a unique $\Delta \widetilde{t}_q < 0$ via solving the following univariate nonlinear equation:
\[x_{i_m}^q(t_q + \Delta {\widetilde{t}_q}  ) - x_{m+1}^q(t_q + \Delta \widetilde{t}_q  ) = 0,\]
which is nothing but the optimality condition associated with the following univariate strongly convex optimization problem:
\[
t_q +  \Delta \widetilde{t}_q  = \argmin_{t} \left\{  (\sum_{i \in  V_{B^{q}}} f_i)^*( t + \beta^{q} ) + (f_{m+1}^*)(-t)\right\}.
\]
The existence of $\Delta \widetilde{t}_q$ is thus guaranteed. Then, we set
\begin{equation}
\label{eq:t_update}
t_{q+1} = \max\{ t_q + \Delta t_q,  t_q + \Delta \widetilde{t}_q  ,  -\lambda_{i_m, m+1} \}.
\end{equation}
As one can observe, it always holds that $t_{q+1} \in [t_{q}  +\Delta t_q, t_q]$ and  
\begin{equation}
\label{eq:xstar_tqp1}
\begin{aligned}
x_{i_m}^*(t_{q+1}) - x_{m+1}^*(t_{q+1}) = {}&x_{i_m}^q(t_{q+1}) - x_{m+1}^q(t_{q+1})\\[2pt]
\ge{}&  x_{i_m}^q(t_{q} + \Delta \widetilde{t}_q ) - x_{m+1}^q(t_{q} + \Delta \widetilde{t}_q ) = 0.
\end{aligned}
\end{equation}
Then, we reveal the relation between  $t_{q+1}$ and $t^*$ in the following lemma.

\begin{lemma}
	\label{lem:opt_tqp1}
	It holds that $-\lambda_{i_m, m+1} \le t^* \le t_{q+1} \le t_q \le 0$. Moreover, $t_{q+1} = t^*$ if and only if $x_{i_m}^*(t_{q+1}) - x_{m+1}^*(t_{q+1}) = 0$ or $t_{q+1} = -\lambda_{i_m, m+1}$.
\end{lemma}
\begin{proof}
	If $t^*> t_{q+1}$, we have from \eqref{eq:t_update} that $t^* > t_{q+1} \ge -\lambda_{i_m,m+1}$. 
	It then follows from \eqref{eq:t_star} that \[x_{i_m}^q(t^*) - x^q_{m+1}(t^*) = x_{i_m}^*(t^*) - x^*_{m+1}(t^*) = 0.\]
	However, we know from \eqref{eq:xstar_tqp1} and the strict monotonicity of $x^q_{i_m}(t) - x^q_{m+1}(t)$ that
	\[
		x_{i_m}^q(t^*) - x^q_{m+1}(t^*) > x_{i_m}^q(t_{q+1}) - x^q_{m+1}(t_{q+1}) \ge 0
	.\] We arrive at a contradiction. Thus, $t^* \le t_{q+1}$.

	Next, if $x_{i_m}^*(t_{q+1}) - x_{m+1}^*(t_{q+1}) = 0$ or $t_{q+1} = -\lambda_{i_m, m+1}$, one can easily verify that $t_{q+1}$, $x_{i_m}^*(t_{q+1})$ and  $x_{m+1}^*(t_{q+1})$ satisfy \eqref{eq:t_star}, i.e., $t^* = t_{q+1}$. Conversely, if $t^* = t_{q+1}$, we have $t_{q+1}\ge -\lambda_{i_m, m+1}$. If $t_{q+1} > -\lambda_{i_m, m+1}$, it follows directly from \eqref{eq:t_star} that 
	$x_{i_m}^*(t^*) - x^*_{m+1}(t^*) = 0$. We thus complete the proof of the lemma.
\end{proof}

\begin{remark}
It is only necessary to compute $\Delta \widetilde{t}_q$ at most once during the entire search process for $t^*$.
Indeed, let
\begin{equation*}
	\Delta_* := \left\{
	\begin{aligned}
		& x_{i_m}^q(t_q + \Delta t_q) - x_{m+1}^q(t_q + \Delta t_q), \quad \mbox{if } \; \Delta t_q > -\infty, \\[2pt]
		& -\infty, \quad \mbox{otherwise.}
	\end{aligned}
	\right.
\end{equation*}
If $\Delta_* \ge 0$, then by  the strict monotonicity of $x^q_{i_m}(t) - x^q_{m+1}(t)$, we must have $\Delta \widetilde{t}_q  \le \Delta t_q$.
In this case, we can directly set
\begin{equation*}
t_{q+1} = 	\max\{ t_q + \Delta t_q, -\lambda_{i_m, m+1} \},
\end{equation*}
without computing $\Delta \widetilde{t}_q $. 
Only when $\Delta_* < 0$, we shall compute $\Delta \widetilde{t}_q $ and set 
\[
	t_{q+1} = 	\max\{ t_q + \Delta \widetilde{t}_q , -\lambda_{i_m, m+1} \}.
\]
Then, from Lemma \ref{lem:opt_tqp1}, it holds that  $t_{q+1} = t^*$.
Therefore, $\Delta \widetilde{t}_q $ only needs to be computed at most once. 
\end{remark}

If $t_{q+1} \neq t^*$, we know from \eqref{eq:t_update}, \eqref{eq:xstar_tqp1}, and Lemma \ref{lem:opt_tqp1} that $t^* < t_{q+1}$ and 
\begin{equation}\label{eq:tqp1}
 t_{q+1} = t_q + \Delta t_q, \quad \mbox{and } \quad x_{i_m}^*(t_{q+1}) - x_{m+1}^*(t_{q+1}) > 0.
\end{equation}
Then, we give the details of the construction of ${\cal A}^{q+1}$.
Let $\mathcal{M}(E_{B^{q}}) = \mathcal{M}(E_{B^{q}}^+) \cup \mathcal{M}(E_{B^{q}}^-)$ with
\begin{equation}
\label{eq:MEBqpm}
\begin{cases}
\mathcal{M}(E_{B^{q}}^+) \;=\; \{ (i,j) \in E_{B^{q}}  \mid \Delta t_{i,j} = \Delta t_q, \; \mbox{and} \; i \triangleleft j \},\\
\mathcal{M}(E_{B^{q}}^-) \;=\; \{ (i,j) \in E_{B^{q}}  \mid \Delta t_{i,j} = \Delta t_q, \; \mbox{and} \; j \triangleleft i \},\\
\end{cases}
\end{equation}
and $\mathcal{M}(\Omega^q) =  \mathcal{M}(\Omega^q_+) \cup \mathcal{M}(\Omega^q_-)$ with 
\begin{equation}
	\label{eq:Momega}
\begin{cases}
\mathcal{M}(\Omega^q_+) ={} \{ (i,j) \in \Omega^q_+ \cap {\cal A}(t_q)_> \mid x^q_{i}(t_q + \Delta t_q ) = x^*_j(t_q) \},  \\
\mathcal{M}(\Omega^q_-) ={} \{ (i,j) \in \Omega^q_- \cap {\cal A}(t_q)_< \mid x^q_{j}(t_q + \Delta t_q ) = x^*_i(t_q) \}. \\
\end{cases}
\end{equation}
The active set $\mathcal{A}^{q+1}$ is constructed via 
\begin{equation}
\label{eq:update_activeset}
\begin{cases}
\mathcal{A}^{q+1}_{=} ={} \big( \mathcal{A}^{q}_{=} \cup \mathcal{M}(\Omega^q) \big) \backslash \mathcal{M}(E_{B^{q}}),  \\
\mathcal{A}^{q+1}_{>} ={} \big( \mathcal{A}^{q}_{>} \cup \mathcal{M}(E_{B^{q}}^+) \big) \backslash \mathcal{M}(\Omega^q_+),   \\
\mathcal{A}^{q+1}_{<} ={} \big( \mathcal{A}^{q}_{<} \cup \mathcal{M}(E_{B^{q}}^-) \big)\backslash \mathcal{M}( \Omega^q_-).  \\
\end{cases}
\end{equation}

Similar to \eqref{eq:tildez_1}, we can construct $\widetilde{z}^q(t_{q+1})$ from $z^q(t_{q+1})$ as follows:
\begin{equation*}
	\tilde{z}_{i,j}^{q}(t_{q+1}) = \left\{ 
	\begin{aligned}
	& z_{i,j}^{q}(t_{q+1}), \, \mbox{ if }(i,j)\in \mathcal{A}^q_{=}, \\[2pt]
	& z_{i,j}^*(t_q), \, \mbox{ otherwise},
	\end{aligned}
	\right.\quad \forall \, (i,j)\in E_m.
\end{equation*}
Then, we obtain the optimal primal-dual solution pair $(x^*(t_{q+1}), z^*(t_{q+1})) = (x^q(t_{q+1}), \widetilde{z}^q(t_{q+1}) )$ to problem \eqref{prob:gm} with $t = t_{q+1}$.

Next, it can be easily verified from the construction of ${\cal A}^{q+1}$ in \eqref{eq:update_activeset}, and the computation steps of $t_{q+1}$ in \eqref{eq:t_update} that the new active set $\mathcal{A}^{q+1}$ and the primal-dual pair $(x^*(t_{q+1}), z^*(t_{q+1}))$ are compatible. This, together with \eqref{eq:tqp1}, allows us to perform induction on {$q\in \mathbb{N}$} and obtain that for all $q\in \mathbb{N}$, as long as $t_q \ne t^*$, it always holds that ${\cal A}^q$ and $(x^*(t_q),z^*(t_q))$ are compatible and 
\[
	x_{i_m}^*(t_q) - x_{m+1}^*(t_q) > 0.
\]
Therefore, we can iteratively repeat the above searching process, i.e., from $(t_q, x^*(t_q), z^*(t_q),{\cal A}^q)$ to $(t_{q+1}, x^*(t_{q+1}), z^*(t_{q+1}),{\cal A}^{q+1})$, until $t^*$ is obtained. The details of the search process {are} summarized in Algorithm \ref{alg:update_-}.
We name it {the} {\textit{update$^-$}} {subroutine}, since in this case $t^* < 0$. The procedure corresponding to the case with $t^* > 0$, which we termed as {the} \textit{update$^+$} {subroutine}, can be easily adapted from the {\textit{update$^-$}} subroutine. Details of the \textit{update$^+$} subroutine can be found in the Appendix.

\begin{algorithm}[htbp]
\caption{ $(t_{q+1}, x^*(t_{q+1}), z^*(t_{q+1}), \mathcal{A}^{q+1}, t^*) = \mbox{{\bf update}}^{-}(t_q, x^*(t_q), z^*(t_q), \mathcal{A}^{q},\lambda)$}
\label{alg:update_-}
\begin{algorithmic}[1]
\State {{\bf Input}: $(t_q, x^*(t_q), z^*(t_q), \mathcal{A}^{q})$, $\lambda \ge 0$; } 	
\State {Compute $\Delta (E_{B^{q}}),  \Delta(\Omega^q_+), \Delta(\Omega^q_-)$ via definitions \eqref{def:delta_EB}, \eqref{def:delta_O+} and \eqref{def:delta_O-}}
	  \State {$ \Delta(\Omega^q) = \max \{ \Delta(\Omega^q_-) , \Delta(\Omega^q_+) \}$} 
	  \State{$ \Delta t_q = \max \{ \Delta (E_{B^{q}}), \Delta(\Omega^q) \}  $ }
      \State {$\Delta^*= x_{i_m}^q(t_q + \Delta t_q) - x_{m+1}^q(t_q + \Delta t_q) $ }		
	  \If {$\Delta^* \ge 0$ }	
	  \State {$t_{q+1} = \max \{ t_q + \Delta t_q , -\lambda   \}$}
	  \Else
	  \State {$\Delta \widetilde{t}^q = -t_q + \argmin\limits_{t} \;\left\{ (\sum_{i \in V_{B^{q}}} f_i)^*( t + \beta^{q} ) + (f_{m+1}^*)(-t)\right\}$ }
	  \State {$t_{q+1} = \max \{ t_q + \Delta \widetilde{t}^q , -\lambda   \}$ }
	  \EndIf
	  \State{$( x^*(t_{q+1}), z^*(t_{q+1}) ) = (x^{q}( t_{q+1} ), \widetilde{z}^{q}( t_{q+1} ) )$}
	  \If { $t_{q+1} = -\lambda$ \textbf{or} $x_{i_m}^*(t_{q+1}) = x_{m+1}^*(t_{q+1})$ }
			\State{ $t^* = t_{q+1}$ }
			\State{ Let $\mathcal{A}^{q+1} = \{ (i,j, \#) \mid (i,j) \in E_m, \; x_i^*(t_{q+1}) \, \# \, x_j^*(t_{q+1}) \}$  }
		\Else
			\State{ $t^* = \emptyset$}	
			\State{ Update $\mathcal{A}^{q+1}$ from $\mathcal{A}^{q}$ via \eqref{eq:update_activeset}}
		\EndIf  
	\State {{\bf Output}: $(t_{q+1}, x^*(t_{q+1}), z^*(t_{q+1}), \mathcal{A}^{q+1}, t^*)$ }
	\end{algorithmic}
\end{algorithm}

Before presenting the details of the {\it generate} subroutine, we make some key observations about the active set $\mathcal{A}^{q+1}$ in the following lemma.
\begin{lemma}
\label{lem:observations}
For any given {$q \in \mathbb{N}$}, {the following propositions hold:}
\begin{enumerate}[(a)]
\item If $t_{q+1} \neq t^*$, then $\mathcal{A}^{q+1}_{=} \neq \mathcal{A}^{q}_{=}$;
\item If $(i,j) \in  \mathcal{M}(E_{B^{q}})$, then for any $ \widehat{q} \in \mathbb{N}$ with $\widehat{q} > q$ and $t_{\widehat{q}} \neq t^*$, $(i,j) \notin \mathcal{A}^{\widehat{q}}_{=}$.
\end{enumerate}
\end{lemma}
\begin{proof}
We prove (a) first. If $t_{q+1} \neq t^*$, from \eqref{eq:tqp1}, we have $t_{q+1} = t_q + \Delta t_q > t^*$.
Hence,  at least one of the two sets, $\mathcal{M}(E_{B^{q}})$ and $\mathcal{M}(\Omega^q)$, is nonempty.
The desired result thus follows since $\mathcal{A}^{q+1}_{=} = \big(  \mathcal{A}^{q+1}  \cup \mathcal{M}(\Omega^q) \big) \backslash \mathcal{M}(E_{B^{q}})$ and {$\mathcal{M}(E_{B^{q}}) \cap \mathcal{M}(\Omega^q) = \emptyset$}.

Next, we prove (b). 
We first consider the case where $i \triangleleft j$.  If $(i,j)\in {\cal M}(E_{B^q})$ and $i \triangleleft j$, we see from \eqref{eq:delta_inEB}, \eqref{eq:MEBqpm} and \eqref{eq:update_activeset} that 
\[
z_{i,j}^{q}(t_{q} + \Delta t_q) = -\lambda_{i,j}, \mbox{ and } (i,j) \in \mathcal{M}(E^+_{B^{q}}) \subseteq \mathcal{A}^{q+1}_>.
\]   
Since $(i,j) \in  \mathcal{A}^{q+1}_>$, then at least one of $i$ and $j$ is not in $B^{q+1}$, i.e., $(i,j)\notin E_{B^{q+1}}$.
Since $i \triangleleft j$, we have the following two possible cases:
\begin{itemize}

\item[{(\rm i)}]  $j\in B^{q+1}, i \notin B^{q+1} $. In this case we have $(i,j) \in \Omega^{q+1}_- $. Since $(i,j) \in \mathcal{A}^{q+1}_>$, it holds from \eqref{eq:Momega} that $(i,j) \notin \mathcal{M}(\Omega^{q+1}_+)$.
Thus, \eqref{eq:update_activeset} implies that  $(i,j) \in \mathcal{A}^{q+2}_>$.

\item[{(\rm ii)}]  $j \notin B^{q+1}, i \notin B^{q+1}$. From \eqref{eq:Omegaq}, we know that 
$(i,j)\notin \Omega^{q+1}$. Hence, \eqref{eq:Momega} and \eqref{eq:update_activeset} imply that $(i,j) \in \mathcal{A}^{q+2}_>$.
\end{itemize}
Therefore, in both cases, we have  $(i,j) \notin \Omega^{q+2}_{+}$ and $(i,j) \in \mathcal{A}^{q+2}_>$. {By induction, we can prove that $(i,j) \notin \Omega^{\widehat{q}}_{+}$ and $(i,j) \in \mathcal{A}^{\widehat q}_>$ for all $\widehat q>q$. } 

Similarly, for the case with $j \triangleleft i$, we can obtain that $(i,j) \notin \Omega^{\widehat{q}}_{-}$ and $(i,j) \in \mathcal{A}^{\widehat q}_<$ for all $\widehat q>q$. We thus complete the proof.
\end{proof}

With {the} two subroutines {\textit{update$^-$}} and {\textit{update$^+$}} at hand, we are ready to present the details of the \textit{generate} subroutine in Algorithm \ref{alg:generate}. As one can easily observe, the complexity of the \textit{generate} subroutine depends critically on the number of executions of the while-loops (i.e., lines 9-12 and lines 15-18 in Algorithm \ref{alg:generate}).

\begin{algorithm}[htbp]
	\caption{ The \textit{generate} subroutine: $(x^{(m+1)} ,z^{(m+1)} ) = \mbox{\bf generate}( x^{(m)}, z^{(m)} , G_{m+1}  )$ }
	\label{alg:generate}
	\begin{algorithmic}[1]
		\State {{\bf Input:} $x^{(m)} \in \Re^m , z^{(m)} \in \Re^{m-1}, G_{m+1} = (V_{m+1}, E_{m+1})$} 
		\State {Let $ x^*_{i}(0) = x^{(m)}_{i} $ for $i \in V_m$ and $x^*_{m+1}(0) = (f^*_{m+1})'(0)$ }	
		\State {Let $ z^*_{i,j}(0) = z^{(m)}_{i,j} $ for $(i,j) \in E_m$  and $t^* = \emptyset$}	
		\\
		\If{ $f'_{m+1}(x_{i_m}^*(0)) = 0$}
		\State { $t^* = 0$}		
		\ElsIf{$f'_{m+1}(x_{i_m}^*(0)) > 0$} 
		\State{Let $t_0 = 0$, $q = 0$ and $\mathcal{A}^0$ be the active set constructed from $x^*(0)$ as in \eqref{eq:A0}}
		\While{$t^* = \emptyset$}
		\State{$(t_{q+1}, x^*(t_{q+1}), z^*(t_{q+1}), \mathcal{A}^{q+1},t^*)$=\textit{update$^-$}$(t_{q}, x^*(t_{q}), z^*(t_{q}), \mathcal{A}^{q}, \lambda_{i_m, m+1} )$}
		\State{$q = q+1$}
		\EndWhile
		\Else %
		\State{Let $t_0 = 0$, $q = 0$ and $\mathcal{A}^0$ be the active set constructed from $x^*(0)$ as in \eqref{eq:A0}}
		\While{$t^* = \emptyset$}
		\State{$(t_{q+1}, x^*(t_{q+1}), z^*(t_{q+1}), \mathcal{A}^{q+1},t^*)$=\textit{update$^+$}$(t_{q}, x^*(t_{q}), z^*(t_{q}), \mathcal{A}^{q},  \mu_{i_m, m+1} )$}
		\State{$q = q+1$}
		\EndWhile
		\EndIf
		\State {Let $x^{(m+1)} = x^*(t^*)$, $z_{i,j}^{(m+1)} = z^*_{i,j}(t^*)$ for $(i,j) \in E_m$, and $z_{i_m,m+1}^{(m+1)} = t^*$}
        \State {{\bf Return:} $( x^{(m+1)} , z^{(m+1)}) \in \Re^{m+1} \times \Re^{m}$ }
	\end{algorithmic}
\end{algorithm}

\begin{lemma}
	\label{lem:finite}
	The while-loops executed in the {\it generate} subroutine will find $t^*$ in at most $2m-1$ iterations.
	\end{lemma}
	\begin{proof}
	Without loss of generality, we only consider the case $f'_{m+1}(x_{i_m}^*(0)) > 0$, i.e., $t^* < 0$.
	Assume that after $2m-2$ times executions of the while-loops, $t^*$ has not been found. That is, the algorithm generates $\{( t_{i}, x^*(t_{i}), z^*(t_{i}), \mathcal{A}^{i})\}_{i=1}^{2m-2}$ and $t_{i} > t^*$ for all $i=0,\ldots,2m-2$. 
	From  Lemma \ref{lem:observations}(a), we know that 
	\begin{equation}
		\label{eq:qneqqp1}
		\mathcal{A}^q_= \neq \mathcal{A}^{q+1}_=, \quad \forall \, q = 0,\ldots, 2m-3 \red{.}
	\end{equation}

	Next, we note from Lemma \ref{lem:observations}(b) that if some edge $(i,j) \in E_m$ is removed from $\mathcal{A}^q_=$ for some $q$, then $(i,j) \not \in \mathcal{A}^{\widehat q}_=$ for all $2m-2\ge \widehat q \ge q\ge 0$. 
	Therefore, for each edge $(i,j) \in E_m$, it can be added to and removed from $\mathcal{A}^q_=$ for at most once. This, together with \eqref{eq:qneqqp1} and the fact that $|E_m| = m-1$, implies that at $t_{2m-2}$, every edge in $E_m$ has been added to and removed from some $\mathcal{A}^q_{=}$. Thus, $\mathcal{A}^{2m-2}_{=} = \emptyset$, and the sets $\mathcal{A}^{2m-2}_{>}$ and $\mathcal{A}^{2m-2}_{<}$ remain unchanged in the next iterations, i.e., $E_{B^q} = \emptyset$, $\Omega^{2m-2}_+ \cap A^{2m-2}_> = \emptyset$ and $\Omega^{2m-2}_- \cap A^{2m-2}_< = \emptyset$. 
	Therefore, we have $\Delta t_{2m-2} = -\infty$ from its definition in \eqref{eq:dtq}. By \eqref{eq:t_update} and Lemma \ref{lem:opt_tqp1}, we have $t_{2m-1} = t^*$ and complete the proof.
\end{proof}

Lemma \ref{lem:finite} guarantees that $t^*$ can be found by the {\it generate} subroutine efficiently. Along with $t^*$, the $G_{m+1}$-optimal pair $(x^{(m+1)}, z^{(m+1)})$, i.e., the output of the {\it generate} subroutine, is also obtained. We thus naturally obtain the correctness of our Algorithm \ref{alg:outer}.

\begin{theorem}
\label{thm:main}
The output $x^{(n)}$ of Algorithm \ref{alg:outer} is the optimal solution to problem \eqref{prob:tree_gnio}.
\end{theorem}

At the end of this section, we provide a brief analysis of the worst-case complexity of our Algorithm \ref{alg:outer}. 
Here, we assume that for a given strongly convex differentiable function $f$ and $x\in \Re$, the computational complexity of finding $t$ such that $f'(t) = x$ is ${\cal O}(1)$. Then, the computational complexity of {\it update}$^{-}$ (and {\it update}$^+$) is ${\cal O}(m)$. By Lemma \ref{lem:finite}, we see that the computational  complexity of the \textit{generate} subroutine is $\mathcal{O}(m^2)$.
Therefore, the computational complexity of Algorithm \ref{alg:outer} is ${\cal O}(n^3)$.

\section{Conclusion}
\label{sec:conclusion}
In this paper, we focus on the convex isotonic regression problem \eqref{prob:tree_gnio} with  tree-induced generalized order restrictions.
Inspired by the successes of the PAVA, an efficient active-set based recursive approach, {ASRA}, is carefully designed to solve \eqref{prob:tree_gnio}. Under mild assumptions, we show that {ASRA} has a polynomial time computational complexity.

\section{Appendix}
\subsection{The arborescence assumption on $G$}
For the given $G = (V,E)$ in the formulation of problem \eqref{prob:tree_gnio}, let $\widehat{G} = (V, \widehat{E})$ be an arborescence that shares the same underlying graph with $G$.
Therefore, for any $(i,j) \in \widehat{E}$, we have either $(i,j) \in E$ or $(j, i ) \in E$.
Then, for any $(i,j) \in \widehat{E}$, let 
\[
\widehat{\lambda}_{i,j} = \begin{cases}
\lambda_{i,j}, & {\rm if} ~ (i,j) \in E, \\
\mu_{j,i}, & {\rm if} ~ (j,i) \in E, \\
\end{cases}  \quad \mbox{and } \quad 
\widehat{\mu}_{i,j} = \begin{cases}
\mu_{i,j}, & {\rm if} ~ (i,j) \in E, \\
\lambda_{j,i}, & {\rm if} ~ (j,i) \in E. \\
\end{cases}
\]
It can be easily verified that problem \eqref{prob:tree_gnio} is equivalent to the following optimization problem:
\begin{equation*}
\min_{x\in\Re^{V} } \quad \sum_{i \in V } f_i(x_i) + \sum_{ (i,j) \in \widehat{E} } \widehat{\lambda}_{{i,j}} (x_i - x_{j})_{+} + \sum_{ (i,j) \in \widehat{E} } \widehat{\mu}_{i,j} (x_{j} - x_{i})_{+}.
\end{equation*}
Hence, we can assume that the directed tree $G$ in \eqref{prob:tree_gnio} is an arborescence.

\begin{figure}[htbp]
\centering
\begin{minipage}{0.28 \linewidth}
\centering
\begin{tikzpicture}[shorten >= 1pt]
\tikzstyle{vertex} = [circle,fill=black!25,minimum size=17pt,inner sep=0pt]
        \node[vertex] (1) { $1$ };
        \node[vertex] (2) [below left = 1] { $2$ };
        \node[vertex] (3) [below right =  1] { $3$ };
        \node[vertex] (4) [below left of = 3] { $4$ };
        \node[vertex] (5) [below right of = 3] { $5$ };
        \draw[->] (2) -- (1);
        \draw[->] (3) -- (1);
        \draw[->] (3) -- (5);
        \draw[->] (3) -- (4);
\end{tikzpicture} 
\centerline{(a) a directed tree $G$}
\end{minipage}
$\Rightarrow$
\begin{minipage}{0.28 \linewidth}
\centering
\begin{tikzpicture}[shorten >= 1pt]
\tikzstyle{vertex} = [circle,fill=black!25,minimum size=17pt,inner sep=0pt]
        \node[vertex] (1) { $1$ };
        \node[vertex] (2) [below left = 1] { $2$ };
        \node[vertex] (3) [below right =  1] { $3$ };
        \node[vertex] (4) [below left of = 3] { $4$ };
        \node[vertex] (5) [below right of = 3] { $5$ };

        \draw (2) -- (1);
        \draw (3) -- (1);
        \draw (3) -- (5);
        \draw (3) -- (4);
	\end{tikzpicture} 
\centerline{(b) underlying graph of $G$ and $\widehat{G}$}
\end{minipage}
$\Leftarrow$
\begin{minipage}{0.28 \linewidth}
\centering
\begin{tikzpicture}[shorten >= 1pt]
\tikzstyle{vertex} = [circle,fill=black!25,minimum size=17pt,inner sep=0pt]
        \node[vertex] (1) { $1$ };
        \node[vertex] (2) [below left = 1] { $2$ };
        \node[vertex] (3) [below right =  1] { $3$ };
        \node[vertex] (4) [below left of = 3] { $4$ };
        \node[vertex] (5) [below right of = 3] { $5$ };

        \draw[->] (1) -- (2);
        \draw[->] (1) -- (3);
        \draw[->] (3) -- (5);
        \draw[->] (3) -- (4);
	\end{tikzpicture} 
\centerline{(c) an arborescence $\widehat{G}$}
\end{minipage}
\caption{A directed tree $G$ and an arborescence $\widehat{G}$ that share the same underlying graph. }
\label{fig:to_rooted_tree}
\end{figure}
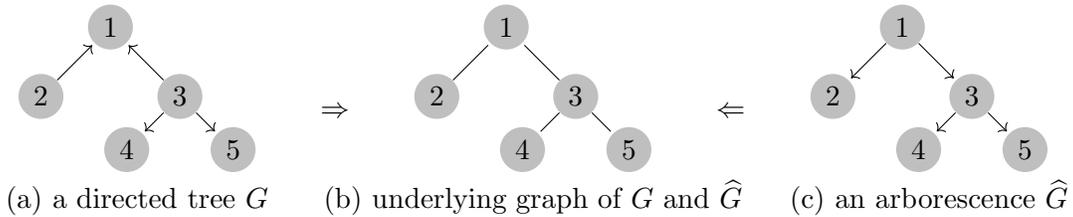

Next, we discuss the decomposition of $G$.
For an arborescence $G = (V,E)$, let $n = |V|$. 
Without loss of generality, we assume that the node $1$ is the root of $G$, and the nodes in $G$ are arranged such that for any edge $(i,j) \in E$, $i < j$ always holds.
Then, we define $G_n = G$, and let $G_{m-1} = (V_{m-1}, E_{m-1})$ be the subgraph of $G_{m} = (V_m , E_m)$ obtained by deleting the node $ m $ and the related edges from $G_m$, where $n \ge m\ge 2$.
Since for any $(i,j) \in E$, it holds that $i < j$, we know that the node $m$ must be a leaf node of $G_m$, hence,
according to \cite{west2001introduction}, $G_{m-1}$ is still a directed tree and $G_{m-1} \subset G_m$ for $m = 2,...,n$.
{It's easy to verify that} $V_m = \{ 1,2,...,m \}$ for $1 \le m \le n$, and $\{ (i_m, m+1) \} = E_{m+1} \backslash E_{m}$ with {$i_m \in V_m$} for $1 \le m \le n-1$.

\subsection{The update$^+$ subroutine}
We briefly describe the {\it update}$^+$ subroutine here, which corresponds to the case with $t^* > 0$.
Assume that we have obtained a guess $t_q$ of $t^*$ satisfying $0 = t_0 \le t_q < t^* \le \mu_{i_m, m+1} $, Meanwhile, the corresponding primal-dual optimal solution pair $(x^*(t_q)$, $z^*(t_q))$ and the active set $\mathcal{A}^q$ are available, such that $\mathcal{A}^q$ and $( x^*(t_q) ,z^*(t_q) ) $ are compatible {and $x^*_{i_m}(t_q) - x^*_{m+1}(t_q) < 0$}.
Then, the semi-closed formulas  \eqref{eq:update_x} and \eqref{eq:update_z} for the $\mathcal{A}^q$-reduced problem in Proposition \ref{prop:closedform} still hold.

Here, we need to search
\[
\Delta t_q := \max \{ \Delta t \mid \eqref{eq:cs_EB} \; \text{and} \; \eqref{eq:zstar_xatq} \; \text{hold for all} \; t \in [t_q, t_q + \Delta t] \}.
\]
First, let
\begin{equation}
\label{def:delta_EB_new}
	\Delta (E_{B^{q}}) = \left\{ \begin{aligned}
	& \min_{(i,j)\in E_{B^{q}}}  \Delta t_{i,j}, \quad \mbox{if } E_{B^{q}}\neq \emptyset, \\[2pt]
   & + \infty, \quad \mbox{otherwise},
\end{aligned}
	\right.
\end{equation}
where each {$\Delta t_{i,j} \ge 0$} solves:
\begin{equation*}
z_{i,j}^{q}(t_q + \Delta t_{i,j}) = { \mu_{i,j}}, \quad \mbox{if } \, i\triangleleft j, \quad \mbox{and} \quad  z_{i,j}^{q}(t_q + \Delta t_{i,j}) = {-\lambda_{i,j}}, \quad \mbox{if } \, j \triangleleft i.
\end{equation*}
Next, let $\Delta(\Omega^{q}) = \min \{ \Delta(\Omega^q_+), \Delta(\Omega^q_-)\}$, where 
\begin{equation}
\label{def:delta_O+_new}
	\Delta(\Omega^q_+) := \left\{
	\begin{aligned}
	& \Delta \overline{t} \mbox{ satisfying } 
	x_{B^q}(t_q + \Delta \overline{t}) = \min_{(i,j)\in \Omega^q_+ \cap {\cal A}^q_<} x^*_j(t_q), \quad  \mbox{if } \; \Omega^q_+ \cap {\cal A}^q_< \neq \emptyset, \\[2pt]
	& +\infty, \mbox{ otherwise,}
	\end{aligned}
	\right.
\end{equation}
and 
\begin{equation}
\label{def:delta_O-_new}
	\Delta(\Omega^q_-):= \left\{
	\begin{aligned}
	& \Delta \overline{t} \mbox{ satisfying } 
	x_{B^q}(t_q + \Delta \overline{t} ) = \min_{(i,j) \in \Omega^q_-  \cap {\cal A}^q_>} x^*_i(t_q), \quad  \mbox{if} \; \Omega^q_- \cap {\cal A}^q_> \neq \emptyset, \\[2pt]
	& +\infty, \mbox{ otherwise.}
	\end{aligned}
	\right.
\end{equation}
Then, $\Delta t_q = \min \{ \Delta (E_{B^q}), \Delta (\Omega^q) \}$.
Compute $\Delta \widetilde{t}_q \ge 0$ via solving $x^q_{i_m}(t_q + \Delta \widetilde{t}_q) - x^q_{m+1}(t_q + \Delta \widetilde{t}_q) = 0$, and set
\[
t_{q+1} = \min \{ t_q + \Delta t_q, t_q + \Delta \widetilde{t}_q, \mu_{i_m, m+1} \}.
\]

If $t_{q+1} < t^*$, we will update the active set $\mathcal{A}^{q+1}$ in the following fashion.
Let $\mathcal{M}(E_{B^{q}}) = \mathcal{M}(E_{B^{q}}^+) \cup \mathcal{M}(E_{B^{q}}^-)$ with
\begin{equation*}
\begin{cases}
\mathcal{M}(E_{B^{q}}^+) \;=\; \{ (i,j) \in E_{B^{q}}  \mid \Delta t_{i,j} = \Delta t_q, \; \mbox{and} \; i \triangleleft j \},\\
\mathcal{M}(E_{B^{q}}^-) \;=\; \{ (i,j) \in E_{B^{q}}  \mid \Delta t_{i,j} = \Delta t_q, \; \mbox{and} \; j \triangleleft i \},\\
\end{cases}
\end{equation*}
and $\mathcal{M}(\Omega^q) = \mathcal{M}(\Omega^q_+) \cup \mathcal{M}(\Omega^q_-)$ with 
\[
\begin{cases}
\mathcal{M}(\Omega^q_+) \;=\; \{ (i,j) \in \Omega^q_+ \cap {\cal A}(t_q)_< \mid x^q_{i}(t_q + \Delta t_q ) = x^*_j(t_q) \},  \\
\mathcal{M}(\Omega^q_-) \;=\; \{ (i,j) \in \Omega^q_- \cap {\cal A}(t_q)_> \mid {x^q_{j}}(t_q + \Delta t_q ) = x^*_i(t_q) \}. \\
\end{cases}
\]
Then, $\mathcal{A}^{q+1}$ is obtained via
\begin{equation}
\label{eq:update_activeset_new}
\begin{cases}
\mathcal{A}^{q+1}_{=} ={} \big( \mathcal{A}^{q}_{=} \cup \mathcal{M}(\Omega^q)   \big)  \backslash \mathcal{M}(E_{B^{q}}),  \\
\mathcal{A}^{q+1}_{>} ={} \big( \mathcal{A}^{q}_{>} \cup \mathcal{M}(E_{B^{q}}^-) \big) \backslash \mathcal{M}( \Omega^q_-),  \\
\mathcal{A}^{q+1}_{<} ={} \big( \mathcal{A}^{q}_{<} \cup \mathcal{M}(E_{B^{q}}^+) \big) \backslash \mathcal{M}(\Omega^q_+).   \\
\end{cases}
\end{equation}
We summarize the {\it update}$^+$ subroutine in Algorithm \ref{alg:update_+}.
\begin{algorithm}[htbp]
	\caption{ $(t_{q+1}, x^*(t_{q+1}), z^*(t_{q+1}), \mathcal{A}^{q+1}, t^*) = \mbox{{\bf update}}^{+}(t_q, x^*(t_q), z^*(t_q), \mathcal{A}^{q},\mu)$}
	\label{alg:update_+}
	\begin{algorithmic}[1]
	\State {{\bf Input}: $(t_q, x^*(t_q), z^*(t_q), \mathcal{A}^{q})$, $\mu \ge 0$; } 	
	\State {Compute $\Delta (E_{B^{q}}),  \Delta(\Omega^q_+), \Delta(\Omega^q_-)$ via definitions \eqref{def:delta_EB_new}, \eqref{def:delta_O+_new} and \eqref{def:delta_O-_new}}
		\State {$ \Delta(\Omega^q) = \min \{ \Delta(\Omega^q_-) , \Delta(\Omega^q_+) \}$} 
		\State{$ \Delta t_q = \min \{ \Delta (E_{B^{q}}), \Delta(\Omega^q) \}  $ }
      \State {$\Delta^*= x_{i_m}^q(t_q + \Delta t_q) - x_{m+1}^q(t_q + \Delta t_q) $ }		
	  \If {$\Delta^* \le 0$ }	
	  \State {$t_{q+1} = \min \{ t_q + \Delta t_q , \mu   \}$}
	  \Else
	  \State {$\Delta \widetilde{t}_q = -t_q + \argmin\limits_{t} \; \left\{ (\sum_{i \in V_{B^{q}}} f_i)^*( t + \beta^{q} ) + (f_{m+1}^*)(-t) \right\}$ }
	  \State {$t_{q+1} = \min \{ t_q + \Delta \widetilde{t}_q , \mu  \}$ }
	  \EndIf
	  \State{$( x^*(t_{q+1}), z^*(t_{q+1}) ) = (x^{q}( t_{q+1} ), \widetilde{z}^{q}( t_{q+1} ) )$}
	  \If { $t_{q+1} = \mu$ \textbf{or} $x_{i_m}^*(t_{q+1}) = x_{m+1}^*(t_{q+1})$ }
			\State{ $t^* = t_{q+1}$ }
			\State{ Let $\mathcal{A}^{q+1} = \{ (i,j, \#) \mid (i,j) \in E_m, \; x_i^*(t_{q+1}) \, \# \, x_j^*(t_{q+1}) \}$ }
		\Else
			\State{ $t^* = \emptyset$}
			\State{ Update $\mathcal{A}^{q+1}$ from $\mathcal{A}^{q}$ via \eqref{eq:update_activeset_new}}
		\EndIf  
			
		\State {{\bf Output}:$(t_{q+1}, x^*(t_{q+1}), z^*(t_{q+1}), \mathcal{A}^{q+1}, t^*)$ }
	\end{algorithmic}
\end{algorithm}

\subsection{An illustration of the ASRA}
We provide an example of applying the ASRA for solving problem \eqref{prob:tree_gnio}.
Let $G = (V, E)$ be the directed tree shown in Figure \ref{fig4a}, where $V = \{ 1,2,3,4,5 \}$ and $E = \{ (1,2), (1,3), (3,4), (3,5) \}$.
Let 
\[
f_i(x_i) = \frac{1}{2} (x_i - y_i)^2 \; {\rm for}\; i=1,...,4, \; {\rm where} \; y = (4,2,2,8) \in \Re^4, \; {\rm and} \; f_5(x_5) = x_5^2  + \frac{1}{4} x_5^4, 
\]
and we set the regularization parameters as follows:
\[
(\lambda_{1,2}, \mu_{1,2}) =(+\infty, 0), \; (\lambda_{1,3},\mu_{1,3}) = (0, +\infty), \; (\lambda_{3,4},\mu_{3,4}) =( 0, 4),  \; {\rm and} \; (\lambda_{3,5}, \mu_{3,5}) = (3,3).
\]
The detailed steps of the ASRA are given below:
\begin{itemize}
\item[{(\rm i)}] First, we initialize {with} $x^{(1)}_1 = 4$.

\item[{(\rm ii)}]  Since $(f^{*}_2)'(x^{(1)}_1) > 0$, it holds that $t^* \le 0$.
We start from $t_0 = 0$ and terminate at $t^* = t_1 = -1$.
Then, the $G_2$-optimal pair $(x^{(2)}, z^{(2)})$ is $x^{(2)} = (3,3)$ and $z^{(2)}_{1,2} = -1$.

\item[{(\rm iii)}] Since $(f^{*}_3)'(x^{(2)}_1) > 0$, we have $t^* \le 0$.
Here, we have $t^* = t_0 = -\lambda_{1,3} = 0$.
The corresponding $G_3$-optimal pair $(x^{(3)}, z^{(3)})$ is $ x^{(3)} = (3,3,2)$, and $ z^{(3)}_{1,2} = -1, z^{(3)}_{1,3} = 0$.

\item[{(\rm iv)}] Since $(f^{*}_4)'(x^{(3)}_3) < 0$, {it holds that} $t^* \ge 0$.
Starting at $t_0 = 0$, we first arrive at $t_1 = 1$, and modify the corresponding active set, i.e., replace $(2,3,>)$ with $(2,3,=)$, then continue the searching of $t^*$.
We terminate at $t^* = t_2 = 4$.
Therefore, the $G_4$-optimal pair $(x^{(4)}, z^{(4)})$ is $x^{(4)} = (4,4,4,4)$, and $z^{(4)}_{1,2} = -2, z^{(4)}_{1,3} = 2, z^{(4)}_{3,4} = 4$.
\item[{(\rm v)}] Since $(f^{*}_5)'(x^{(4)}_3) > 0$, we have $t^* \le 0$.
Starting from $t_0 = 0$, we first arrive $t_1 = 0$ and replace $(3,4,=)$ with $(3,4,<)$ in the corresponding active set.
Then, we terminate the searching at $t^* = t_2 = -3$, and the $G_5$-optimal pair $(x^{(5)}, z^{(5)})$ is  $x^{(5)} = (3,3,3,4,1)$, and $ z^{(5)}_{1,2} = -1, z^{(5)}_{1,3} = 0, z^{(5)}_{3,4} = 4, z^{(5)}_{3,5} = -3$.
\end{itemize}
Thus, the optimal solution to problem \eqref{prob:tree_gnio} is $x^* = (3,3,3,4,1)$. An illustration of the above procedure is presented in Figure \ref{fig:recursion}.

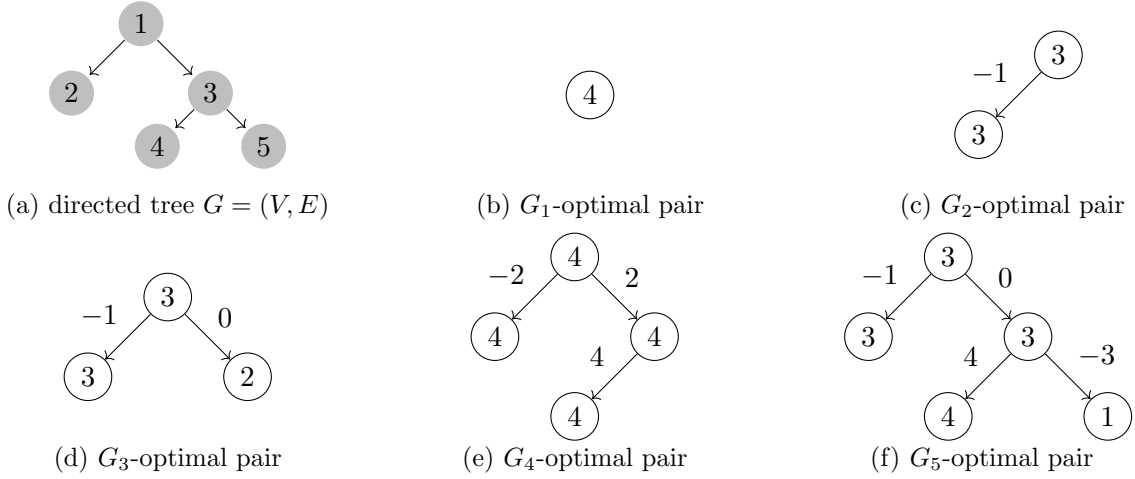
\begin{figure}[]
\begin{subfigure}{0.32 \linewidth}
\parbox[][2.4cm][c]{\linewidth}{
\centering
\begin{tikzpicture}[shorten >= 1pt]
\tikzstyle{vertex} = [circle,fill=black!25,minimum size=17pt,inner sep=0pt]
        \node[vertex] (1) { $1$ };
        \node[vertex] (2) [below left = 1] { $2$ };
        \node[vertex] (3) [below right =  1] { $3$ };
        \node[vertex] (4) [below left of = 3] { $4$ };
        \node[vertex] (5) [below right of = 3] { $5$ };

        \draw[->] (1) -- (2);
        \draw[->] (1) -- (3);
        \draw[->] (3) -- (5);
        \draw[->] (3) -- (4);
\end{tikzpicture} 
}
\caption{directed tree $G =(V,E)$}
\label{fig4a}
\end{subfigure}
\begin{subfigure}{0.32 \linewidth}
\parbox[][2.4cm][c]{\linewidth}{
\centering
\begin{tikzpicture}[node distance = 15mm]
\tikzstyle{vtx} = [draw,circle,minimum size=18pt,inner sep=0pt]
\tikzset{
	dot/.style={rotate= -45,font=\LARGE}, 
}
        \node[vtx, label = above :  ] (1) { $4$ };

\end{tikzpicture} 
}
\caption{$G_1$-optimal pair}
\label{fig4a1}
\end{subfigure}
\begin{subfigure}{0.32 \linewidth}
\parbox[][2.4cm][c]{\linewidth}{
\centering
\begin{tikzpicture}[node distance = 15mm]
\tikzstyle{vtx} = [draw,circle,minimum size=18pt,inner sep=0pt]
\tikzset{
	dot/.style={rotate= -45,font=\LARGE}, 
}
        \node[vtx, label = above :  ] (1) { $3$ };
        \node[vtx, label = above left :  ]  [below left of = 1] (2) {$3$};

        \draw[->] (1) -- node[above left] {$-1$} (2);
\end{tikzpicture} 
}
\caption{$G_2$-optimal pair}
\label{fig4b}
\end{subfigure}
\begin{subfigure}{0.32 \linewidth}
\parbox[][2.7cm][c]{\linewidth}{
\centering
\begin{tikzpicture}[node distance = 15mm]
\tikzstyle{vtx} = [draw,circle,minimum size=18pt,inner sep=0pt]
\tikzset{
	dot/.style={rotate= -45,font=\LARGE}, %
}
        \node[vtx, label = above :  ] (1) { $3$ };
        \node[vtx, label = above left :  ]  [below left of = 1] (2) {$3$};
        \node[vtx, label = above right : ]  [below right of = 1] (3) {$2$};
        \draw[->] (1) -- node[above left] {$-1$} (2);
        \draw[->] (1) -- node[above right] {$0$} (3);

	\end{tikzpicture}  
}
\caption{$G_3$-optimal pair}
\label{fig4c}
\end{subfigure}
\begin{subfigure}{0.32 \linewidth}
\parbox[][2.7cm][c]{\linewidth}{
\centering
\begin{tikzpicture}[node distance = 15mm]
\tikzstyle{vtx} = [draw,circle,minimum size=18pt,inner sep=0pt]
\tikzset{
	dot/.style={rotate= -45,font=\LARGE}, %
}
        \node[vtx, label = above :  ] (1) { $4$ };
        \node[vtx, label = above left :  ]  [below left of = 1] (2) {$4$};
        \node[vtx, label = above right : ]  [below right of = 1] (3) {$4$};
        \node[vtx, label = above : ]  [below left of = 3] (4) {$4$};

        \draw[->] (1) -- node[above left] {$-2$} (2);
        \draw[->] (1) -- node[above right] {$2$} (3);
        \draw[->] (3) -- node[above left] {$4$} (4);

	\end{tikzpicture} 
}
\caption{$G_4$-optimal pair}
\label{fig4d}
\end{subfigure}
\begin{subfigure}{0.32 \linewidth}
\parbox[][2.7cm][c]{\linewidth}{
\centering
\begin{tikzpicture}[node distance = 15mm]
\tikzstyle{vtx} = [draw,circle,minimum size=18pt,inner sep=0pt]
\tikzset{
	dot/.style={rotate= -45,font=\LARGE}, %
}
        \node[vtx, label = above :  ] (1) { $3$ };
        \node[vtx, label = above left :  ]  [below left of = 1] (2) {$3$};
        \node[vtx, label = above right : ]  [below right of = 1] (3) {$3$};
        \node[vtx, label = above : ]  [below left of = 3] (4) {$4$};
        \node[vtx, label = above : ]  [below right of = 3] (5) {$1$};
        \draw[->] (1) -- node[above left] {$-1$} (2);
        \draw[->] (1) -- node[above right] {$0$} (3);
        \draw[->] (3) -- node[above left] {$4$} (4);
        \draw[->] (3) -- node[above right] {$-3$} (5);
	\end{tikzpicture} 
	}
\caption{$G_5$-optimal pair}
\label{fig4e}
\end{subfigure}
\caption{An example of applying {the ASRA} for solving problem \eqref{prob:tree_gnio} with given $G = (V,E)$. The first subfigure represents the directed tree $G = (V,E)$, and the remaining five subfigures are the illustrations of the $G_m$-optimal pairs for $m = 1,2,3,4,5$, where the values of $x_i$ for $i \in V$ are presented within the circles while the values of $z_{i,j}$ for $(i,j)\in E$ are presented above the edges.}
\label{fig:recursion}
\end{figure}

\bibliographystyle{siam}

\end{document}